\immediate\write18{makeindex -s nomencl.ist -o "\jobname.nls" "\jobname.nlo"}

\documentclass[11pt,reqno,tbtags]{amsart}

\usepackage{amsthm,amssymb,amsfonts,amsmath}
\usepackage{amscd} 
\usepackage{mathrsfs}
\usepackage[numbers, sort&compress]{natbib}
\usepackage[lofdepth,lotdepth]{subfig}
\usepackage{graphicx}
\usepackage{vmargin}
\usepackage{setspace}
\usepackage{paralist}
\usepackage[pagebackref=true]{hyperref}
\usepackage{color}
\usepackage[normalem]{ulem}
\usepackage{stmaryrd} 
\usepackage{bm} 
\usepackage[refpage,noprefix,intoc]{nomencl} 

\providecommand{\R}{}

\providecommand{\N}{}

\renewcommand{\R}{\mathbb{R}}

\renewcommand{\N}{{\mathbb N}}

\newcommand{\E}[1]{{\mathbf E}\left[#1\right]}

\newcommand{\p}[1]{{\mathbf P}\left\{#1\right\}}

\newcommand{\set}[1]{\left\{ #1 \right\}}

\newcommand{\probC}[2]{\mathbf{P}\set{#1 \; \left|  \; #2 \right. }}

 \newcommand{\bag}{\begin{align}}
\newcommand{\bags}{\begin{align*}}
\newcommand{\eag}{\end{align*}}
\newcommand{\eags}{\end{align*}}

\newtheorem{thm}{Theorem}
\newtheorem{lem}[thm]{Lemma}
\newtheorem{prop}[thm]{Proposition}

\newtheorem{definition}[thm]{Definition}

\newtheorem{claim}[thm]{Claim}

\theoremstyle{definition}
\newtheorem{assumption}{Assumption}
\theoremstyle{remark}
\newtheorem{remark}{Remark}

\newcommand\cL{{\mathcal L}}
\newcommand\cM{\mathcal M}

\newcommand\cT{{\mathcal T}}

\newcommand\cZ{{\mathcal Z}}

\newcommand{\rD}{\mathrm{D}}

\newcommand{\rQ}{\mathrm{Q}} 
\newcommand{\rR}{\mathrm{R}} 
 
\newcommand{\rT}{\mathrm{T}}

\newcommand{\bM}{\mathbf{M}}

\newcommand{\bT}{\mathbf{T}}

\providecommand{\eps}{}
\renewcommand{\eps}{\epsilon}
\providecommand{\ora}[1]{}
\renewcommand{\ora}[1]{\overrightarrow{#1}}

\newcommand{\bbr}[1]{\ensuremath{\llbracket #1 \rrbracket}} 

\newcommand{\eqdist}{\ensuremath{\stackrel{\mathrm{d}}{=}}}
\newcommand{\convdist}{\ensuremath{\stackrel{\mathrm{d}}{\rightarrow}}}

\newcommand{\convpr}{\ensuremath{\stackrel{\mathrm{(p)}}{\longrightarrow}}}
\newcommand{\aseq}{\ensuremath{\stackrel{\mathrm{a.s.}}{=}}}

\hypersetup{
    bookmarks=true,         
    unicode=false,          
    pdftoolbar=true,        
    pdfmenubar=true,        
    pdffitwindow=true,      
    pdftitle={My title},    
    pdfauthor={Author},     
    pdfsubject={Subject},   
    pdfnewwindow=true,      
    pdfkeywords={keywords}, 
    colorlinks=true,       
    linkcolor=blue,          
    citecolor=blue,        
    filecolor=blue,      
    urlcolor=blue           
}

\reversemarginpar
\definecolor{clou}{rgb}{0.8,0.25,0.5125}

\newcommand\urladdrx[1]{{\urladdr{\def~{{\tiny$\sim$}}#1}}}

\begingroup
  \divide\count255 by 60
  \multiply\count255 by -60
  \advance\count255 by \time
  \ifnum \count255 < 10 \xdef\oclock{\the\count1:0\the\count255}
  \else\xdef\oclock{\the\count1:\the\count255}\fi
\endgroup

\makenomenclature											
\DeclareRobustCommand{\SkipTocEntry}[5]{}
\newcommand{\sym}{\ensuremath{\mathrm{sym}}}
\newcommand{\angles}[1]{\langle #1 \rangle}
\newcommand{\rv}{\ensuremath{\mathrm{v}}}
\newcommand{\rd}{\ensuremath{\mathrm{d}}}
\newcommand{\rr}{\ensuremath{\mathrm{r}}}
\newcommand{\rs}{\ensuremath{\mathrm{s}}}

\newcommand{\rz}{\ensuremath{\mathrm{z}}}
\newcommand{\ru}{\ensuremath{\mathrm{u}}}

\newcommand{\rt}{\ensuremath{\mathrm{t}}}
\newcommand{\dist}{\ensuremath{\mathrm{dist}}}
\newcommand{\Dist}{\ensuremath{\mathrm{Dist}}}

\newcommand{\branch}{\ensuremath{\mathrm{branch
}}}
\newcommand{\cht}{\ensuremath{\mathrm{ctype}}}
\providecommand{deg}{}
\renewcommand{\deg}{\ensuremath{\mathrm{deg}}}

\newcommand{\be}{\ensuremath{\mathbf{e}}}

\newcommand{\lip}[1]{\ensuremath{|\!|#1|\!|_{\text{Lip}}}}
\newcommand{\ninf}[1]{\ensuremath{|\!|#1|\!|_{\infty}}}

\newcommand{\limC}{\ensuremath{C}}
\newcommand{\limZ}{\ensuremath{Z}}
\newcommand{\sbt}{{\ensuremath{\scriptscriptstyle\bullet}}}

\newcommand{\Permvecs}{P}
\newcommand{\unif}{\ensuremath{X}}
\newcommand{\dfun}{\ensuremath{\delta}}

\begin{document}
\title{Convergence of non-bipartite maps via symmetrization of labeled trees}
\author{Louigi Addario-Berry \and Marie Albenque}
\address{Department of Mathematics and Statistics, McGill University, 805 Sherbrooke Street West, 
		Montr\'eal, Qu\'ebec, H3A 2K6, Canada}
\address{LIX UMR~7161, \'Ecole Polytechnique, 1 Rue Honor\'e d'Estienne d'Orves, 91120 Palaiseau, France.}
\email{louigi.addario@mcgill.ca}
\email{albenque@lix.polytechnique.fr}
\date{April 9, 2019} 
\urladdrx{http://problab.ca/louigi/}
\urladdrx{http://www.lix.polytechnique.fr/~albenque/}


\begin{abstract} 
Fix an odd integer $p\geq 5$. Let $M_n$ be a uniform $p$-angulation with $n$ vertices, endowed with the uniform probability measure on its vertices. We prove that there exists $C_p\in \mathbb{R}_+$ such that, after rescaling distances by $C_p/n^{1/4}$, $M_n$ converges in distribution for the Gromov-Hausdorff-Prokhorov topology towards the Brownian map. To prove the preceding fact, we introduce a bootstrapping principle for distributional convergence of random labelled plane trees. In particular, the latter allows to obtain an invariance principle for labeled multitype Galton-Watson trees, with only a weak assumption on the centering of label displacements. 
\end{abstract}

\maketitle

\section{Introduction}\label{sec:introduction} 

\subsection{Convergence of random planar maps}
A \emph{planar map} is an embedding of a finite connected graph into the two-dimensional sphere, viewed up to orientation-preserving homeomorphisms. For $p\geq 3$, a $p$-angulation is a planar map whose faces all have degree $p$. Scaling limits of random planar maps have been the subject of a lot of attention in recent years; perhaps the most celebrated results are the independent proofs by Miermont \cite{MiermontBrownian} and Le Gall \cite{LeGallUniqueness} of the fact that the scaling limit of random 4-angulations (or quadrangulations) is the Brownian map. In fact, in his work Le Gall also established that, for $p=3$ or $p\ge 4$ even, the scaling limit of $p$-angulations is  the Brownian map. The current paper establishes the analogous result for $p$-angulations with $p\ge 5$ odd.
\begin{thm}\label{th:oddAngulations}
Let $p\geq 5$ be an \emph{odd} integer and let $(\bM_n)$ be a sequence of independent random maps, such that for any $n\geq 1$, $\bM_n$ is a uniform $p$-angulation with $n$ vertices. Denote by $\dist_{\bM_n}$ the graph distance on $\bM_n$ and $\mu_n$ the uniform probability distribution on its set of vertices $V(\bM_n)$. Then there exists a constant $C_p$ such that, as $n$ goes to infinity, 
\[
	\Big(V(\bM_n),\frac{C_p}{n^{1/4}}\dist_{\bM_n},\mu_n\Big) \convdist (M,d^\star,\lambda),
\]
for the Gromov-Hausdorff-Prokhorov topology and where $(M,d^\star,\lambda)$ is the Brownian map.
\end{thm}
We obtain Theorem~\ref{th:oddAngulations} as a consequence of a more general result, stated in Theorem~\ref{th:convergence}, which establishes convergence to the Brownian map for so-called \emph{regular critical Boltzmann maps}. Before giving further context for our result, and the ideas of its proof, let us first emphasize that it relies on the work of Miermont and Le Gall and does not constitute an independent proof of the uniqueness of the limiting object. 

The main motivation for our work is the conjecture that the Brownian map is a universal limiting object for many families of planar maps. As already mentioned, the Brownian map is known to be the scaling limit for uniform $p$-angulations for $p\in \{3\}\cup 2\mathbb{N}$, but also for quadrangulations without vertices of degree one~\cite{beltran13quad}, for simple triangulations and quadrangulations~\cite{AddarioAlbenqueSimple}, for general maps~\cite{BettinelliJacobMiermont}, for bipartite maps~\cite{AbrahamBipartite} and for bipartite maps with prescribed degree sequence~\cite{MarzoukPrescribed}. In this sense, our result is an additional step towards the universality of the Brownian map. Moreover, recall that a \emph{bipartite} map is a map whose vertices can be partitioned into two sets, say $B$ and $W$ such that all edges in the map have one extremity in $B$ and one extremity in $W$. It is easy to see that a planar map is bipartite if and only if all its faces have even degree. 
With the notable exception of \cite{BettinelliJacobMiermont} (which does not control the degree of faces of the maps considered), all the results listed above deal with either bipartite maps or with triangulations.
Some results about distance statistics of odd-angulations (and more generally of non-bipartite regular critical Boltzmann maps) have been obtained previously in~\cite{MiermontInvariance,MiermontWeill}, but the methods developed in these works did not yield convergence to the Brownian map.

As with most results in this field, our work relies on a bijection between planar maps and labeled multitype trees: the Bouttier-di Francesco-Guitter bijection~\cite{BDG2004} plays this role in our case. Thanks to the general approach developed in \cite{LeGallUniqueness}, the only new result needed to prove Theorem~\ref{th:oddAngulations} is that the encoding functions of multitype labeled trees associated to $p$-angulations by that bijection converge to the Brownian snake. Numerous results about convergence of labeled trees already exist in the litterature (see for instance~\cite{MiermontInvarianceTrees}, \cite{MarckertLineage}). However, most of these results rely on the assumption that the variation of labels along an edge is centered (see Section~\ref{sub:centering} below), or that degrees are bounded and the trees only have one type; such assumptions do not hold in our setting. To describe how we circumvent this difficulty, we introduce some further notations and definitions. 

\subsection{Symmetrization of labeled trees}
Let $t$ be a rooted plane tree. For a vertex $v$ of $t$ we write $k_t(v)$ 
\nomenclature[ktv]{$k_t(v)$}{The number of children of $v$ in $t$.}
for the number of children of $v$ in $t$ ($k$ stands for ``kids''). 
In the following we identify the vertex set $V(t)$ with the set of words given by the Ulam-Harris encoding. In this encoding, nodes are labeled by elements of $\bigcup_{n \ge 0} \N^n$, where $\N^0=\{\emptyset\}$ by convention. The root receives Ulam-Harris label $\emptyset$; the children of node $v=v_1v_2\ldots v_h \in \N^h$ receive Ulam-Harris labels $(vi,1 \le i \le k_t(v))$ in the order given by the plane embedding.

A \emph{rooted labeled plane tree} is a pair $\rt=(t,\rd)$ where $t$ is a rooted plane tree and $\rd=(d(e),e \in E(t)) \in \R^{E(t)}$ 
\nomenclature[de]{$d(e)$}{The displacement along edge $e$.}
give the \emph{displacements} of labels along the edges $E(t)$ of $t$.

We fix, for the remainder of the paper, a generic countable set $S$ which will serve as a type space. A tree $t$ is a multitype tree if each node $v \in V(t)$ has a type $s(v) \in S$.
\nomenclature[s]{$s(v)$}{The type of vertex $v$; $s(v) \in S$.}
We likewise define multitype labeled plane trees. We may view an un-typed tree as a typed tree by giving all nodes the same type, so hereafter all trees in the paper are considered to be multitype, unless we mention otherwise.

\smallskip
Given a plane tree $t$, let $[t]$ 
\nomenclature[brackett]{$[t]$}{Set of plane trees isomorphic to $t$ as multitype rooted trees.}
be the set of plane trees which are isomorphic to $t$ as multitype rooted trees (but not necessarily as multitype \emph{plane} rooted trees). 
Write $\Permvecs_t$ 
\nomenclature[Pt]{$\Permvecs_t$}{Set of vectors $(\sigma_v,v\in V(t))$, each $\sigma_v \in \mathfrak{S}_{k_t(v)}$}
for the set of vectors $\sigma=(\sigma_v, v \in V(t))$, where each $\sigma_v$ is a permutation of $\{1,\ldots,k_t(v)\}$. Such a vector $\sigma$ uniquely specifies a tree $t'=\sigma(t) \in [t]$
\nomenclature[sigmat]{$\sigma(t)$}{Tree $t' \in [t]$ obtained by reordering children of each node $v$ using $\sigma_v$.} by reordering the children at each node according to $\sigma$ as follows: 
for each node $v=v_1v_2\ldots v_k \in V(t)$, there is a corresponding node $\sigma(v) \in V(t')$
\nomenclature[sigmav]{$\sigma(v)$}{Image of $v$ in $\sigma(t)$}
whose type is the same as that of $v$ and whose Ulam-Harris label is
\[
\sigma(v)=\sigma_\emptyset(v_1)\sigma_{v_1}(v_2)\ldots \sigma_{v_1\ldots v_{k-1}}(v_k)\, .
\]
Visually, this reorders the children of each node $v$ according to the permutation $\sigma_v$.  

If $\rt$ is a labeled plane tree, we likewise define $[\rt]$ 
\nomenclature[bracketttrom]{$[\rt]$}{Set of plane trees isomorphic to $\rt$ as multitype labeled rooted trees.}
and $\rt' = \sigma(\rt) \in [\rt]$ 
\nomenclature[sigmatrom]{$\sigma(\rt)$}{Tree $\rt' \in [\rt]$ obtained by reordering children, child edge labels of $v$ using $\sigma_v$.}
by letting the labels follow their edges. Formally, if 
$e=uv \in e(t)$ then $d'(\sigma(u)\sigma(v)) = d(uv)$.

We typically use $\mu$ to denote a measure on unlabeled plane trees. We say such a measure $\mu$ is symmetric
\nomenclature[Symmetric]{Symmetric}{Measure $\mu$ (resp.\ $\nu$) on plane trees (resp.\ labeled plane trees) is symmetric if $\mu(t)=\mu(t')$ whenever $t' \in [t]$ (resp.\ if $\nu(\rt)=\nu(\rt')$ whenever $\rt'\in [\rt]$).}
 if $\mu(t)=\mu(t')$ for all plane trees $t,t'$ with $t'\in [t]$. Similarly, $\nu$ will typically denote a measure on labelled plane trees, and we say such a measure $\nu$ is symmetric if $\nu((t,d))=\nu((t',d'))$ whenever $(t',d') \in [(t,d)]$. 

Fix a plane tree $t$ and let $\sigma = (\sigma_v,v \in V(t)) \in_u \Permvecs_t$ be a uniformly random element of $\Permvecs_t$. We call the random tree $\sigma(t)$ the {\em symmetrization} of $t$. This also makes sense for a random tree $T$; in this case, conditionally given $T$, we have $\sigma \in_u \Permvecs_T$, and the symmetrization is the tree $\sigma(T)$. If $\mu$ is the law of $T$ then we write $\mu^\sym$ for the law of its symmetrization. Note that $\mu$ is symmetric if and only if $\mu=\mu^\sym$. 
\nomenclature[musym]{$\mu^\sym$}{The symmetrization of measure $\mu$.}

The definitions of the preceding paragraphs all have analogues for labeled plane trees. The symmetrization of $(t,\rd)$, is $\sigma(t,\rd)$, where $\sigma \in_u \Permvecs_t$; if $\nu$ is the law of random labeled plane tree $(T,\rD)$ then we write $\nu^{\sym}$ 
\nomenclature[nusym]{$\nu^\sym$}{The symmetrization of measure $\nu$.}
for the law of the symmetrization of $(T,\rD)$; and, $\nu$ is symmetric if and only if $\nu=\nu^\sym$.
\medskip

This work establishes a tool for establishing distributional convergence of random labelled plane trees, if such convergence is already known for symmetrized versions of the trees. We exclusively consider random labeled plane trees $(T,\rD)$ satisfying the following three properties. 
\begin{enumerate}
\item[(i)] The law $\mu$ of the underlying unlabeled plane tree is symmetric.
\item[(ii)] For each $v \in V(T)$, let $D_v = (D_{v,v1},\ldots,D_{v,vk(v)})$ be the vector of displacements from $v$ to its children. Then the vectors $(D_v,v \in V(T))$ are conditionally independent given $T$. 
\item[(iii)] The law of each vector $D_v$ is determined by the type of $v$ together with the vector of types of its children. 
\end{enumerate}
If $(T,\rD)$ satisfies all three properties then we say its law $\nu$ is {\em valid}. 
\nomenclature[Validlaw]{Valid law}{A law $\nu$ on labeled plane trees satisfying certain independence and symmetry properties.}
\smallskip

The following theorem establishes that certain asymptotic distributional properties of random labeled plane trees sampled from a valid law are unchanged by symmetrization of the distributions involved. The proof makes reference to the contour and label processes of trees; these are defined in Section~\ref{sec:fdd}, immediately following the statement of the theorem.

\begin{thm}\label{thm:main}
For each $n \ge 1$ let $\rT_n=(T_n,\rD_n)$ be a random labeled multitype plane tree whose law $\nu_n$ is valid, and let $\rT_n^\sym$ have law $\nu_n^\sym$.  Write $(C_{\rT_n}(t),Z_{\rT_n}(t))_{0\leq t\leq 1}$ and $(C_{\rT_n^{\mathrm{sym}}}(t),Z_{\rT_n^{\mathrm{sym}}}(t))_{0\leq t\leq 1}$ for the contour and label processes of $\rT_n$ and $\rT_n^{\mathrm{sym}}$, respectively. 

Suppose there exist positive sequences $(a_n,n \ge 1)$ and $(b_n,n \ge 1)$ with $a_n \to 0$ and $b_n \to 0$, and a $C([0,1],\R^2)$-valued random process $(C,Z)=((C(t),Z(t)),t \in [0,1])$, such that 
\[
(a_nC_{\rT_n^{\mathrm{sym}}},b_nZ_{\rT_n^{\mathrm{sym}}}) \convdist (C,Z)\]
 for the uniform topology. 
 Then, the following convergence also holds for the uniform topology:
 \[
 (a_nC_{\rT_n},b_nZ_{\rT_n}) \convdist (C,Z).\]
\end{thm}

Let us conclude by coming back to planar maps. For $p\geq 5$ and $p\notin 2\mathbb{N}$, the push-forward of the uniform measure on p-angulations (or indeed of any regular critical Boltzmann distribution on maps) by the Bouttier-di Francesco-Guitter (described in Section~\ref{sub:BDG}, below) is a probability distribution on labeled plane trees which is valid but not symmetric. Theorem~\ref{thm:main} allows the transfer of results of Miermont~\cite{MiermontInvarianceTrees} to establish that contour and label processes of these trees have the same scaling limit as their symmetrized versions (see Theorem~\ref{thm:convSnake} for a more precise statement).

\section{Valid laws, finite-dimensional-distributions and symmetrization}\label{sec:fdd}
\subsection{Definitions and notation}\label{sub:def}
Recall that throughout the paper, we use $S$ to denote the type space of our trees. We write $S^{\mathrm{fin}}$ for the set of all vectors of finite length with entries from $S$;
we include the empty vector $()$ of length zero in this set. 

In this section $t$ always denotes a rooted multitype plane tree and $\rt$ a labeled rooted multitype plane tree. 
For a vertex $v \in V(t)$, the {\em child type vector} of $v$ is the vector $\cht(v) = (s(vi),1 \le i \le k_t(v)) \in S^{k_t(v)}\subset S^\mathrm{fin}$. 
\nomenclature[ctv]{$\cht(v)$}{The child type vector of $v$.}

\smallskip
We use the notation $|\rt|=|t|=|V(t)|$ interchangeably. 
We define the contour exploration $\theta=\theta_t:\{0,\ldots,2|t|-2\}\rightarrow V(t)$
\nomenclature[thetat]{$\theta_t$}{The contour exploration of tree $t$; takes values in $V(t)$.}
 of $t$ inductively as follows. Let $\theta(0) = \emptyset$.  Then, for $1\leq i \leq 2|t|-2$, let $\theta(i)$ be the lexicographically first child of $\theta(i-1)$ that is not an element of $\{\theta(0),\ldots,\theta(i-1)\}$, or let $\theta(i)$ be the parent of $\theta(i-1)$ if no such node exists. 

For $v,w \in V(t)$, we write $\dist(v,w)=\dist_t(v,w)$
\nomenclature[distvw]{$\dist_t(v,w)$}{Graph distance between $v$ and $w$ in $t$.}
 for the graph distance between $v$ and $w$ in $t$. We write 
$h(v)=\dist_t(\emptyset,v)$
\nomenclature[hv]{$h(v)$}{The distance $\dist(\emptyset,v)$ from the root to $v$.}
 for the distance from $v$ to  the root, so if $v \in \N^k$ then $h(v) = k$. 
Note that $h(v)$ is the length of $v$ in the Ulam-Harris encoding. 
The label of a vertex $v$ of $\rt=(t,\rd)$, denoted $\ell_\rt(v)$ or $\ell(v)$,
\nomenclature[lv]{$\ell(v)$}{Sum of displacements on root-to-$v$ path.}
 is defined as the sum of the edge labels on the path between $v$ and the root. Formally, if $v=v_1\ldots v_k$ then $\ell_\rt(v)=d(\emptyset,v_1)+\ldots+d(v_1\ldots v_{k-1},v_1\ldots v_k$). 

The contour and label processes $C_\rt$
\nomenclature[Ct]{$C_\rt$}{The contour process of $\rt$; $C_{\rt}:[0,1] \to \R$.}
 and $Z_\rt$ 
\nomenclature[Zt]{$Z_\rt$}{The label process of $\rt$; $Z_{\rt}:[0,1] \to \R$.}
 are the functions from $[0,1]$ to $\R$ defined as follows. For $0\leq i \leq n$, set $C_\rt(i/(2|t|-2))=h(\theta(i))$ and $Z_\rt(i/(2|t|-2))=\ell(\theta(i))$. (Note that $C_\rt$ may be viewed as traversing an edge in time $1/(2|t|-2)$; this differs from some other works, where the contour process explores each edge in one unit of time.). 

\subsection{Notes on treelike paths}\label{sub:treelike}
The metric structure of $\rt$ may be recovered from $C=C_\rt$ as follows. For $x,y \in [0,1]$ with $x \le y$, let 
\[
\Dist_C(x,y) =\Dist_C(y,x) = C_\rt(x)+C_\rt(y)-2\inf\{C_\rt(u), x\le u\le y\}. 
\]
\nomenclature[Distcxy]{$\Dist_C(x,y)$}{$\Dist_C(x,y) = C_\rt(x)+C_\rt(y)-2\inf\{C_\rt(u), x\le u\le y\}$}
Then for all $0 \le i,j \le n$, 
$\dist_t(\theta_t(i),\theta_t(j)) = \Dist_C(i/n,j/n)$. In particular, 
if $\Dist_C(i/n,j/n) = 0$ then $\theta_t(i)=\theta_t(j)$. 
Let $\sim_C$ be the equivalence relation $\{(x,y): \Dist_C(x,y)=0\}$, and let $\cT_C = [0,1]/{\sim_C}$.  Then  
$(\cT_C,\Dist_C)$ is a metric space (here, by $\Dist_C$ we really mean its push-forward to $\cT_C$), and its subspace induced by the points $\{[i/n],0 \le i \le n\}$ is isometric to $t$. 

Using the equivalence of $\dist_t$ and $\Dist_C$ at lattice times, if $\Dist_C(i/n,j/n)=0$ then $Z_\rt(i/n)=Z_\rt(j/n)$. Since $Z_\rt$ is defined by linear interpolation, it follows that $Z_\rt(x)=Z_\rt(y)$ whenever $\Dist_C(x,y)=0$, and that the push-forward $Z$ of $Z_\rt$ to $\cT_C$ is well-defined and continuous for $\Dist_C$. We then clearly have $Z([i/n])=\ell_\rt(\theta_t(i))$ for all $0 \le i \le n$.

The two preceding paragraphs may be viewed as motivation for the following general construction. Say $\zeta \in C([0,1],\R^+)$ is an excursion if $\zeta(0)=\zeta(1)=0$. If $\zeta$ is an excursion then we may define $\Dist_\zeta$ and $\cT_\zeta$ just as above, and $(\cT_\zeta,\Dist_\zeta)$ is always a compact metric space. (In fact, it is always an $\R$-tree.)

Now fix a pair $(\zeta,f)$ with $\zeta \in C([0,1],\R^+)$ and $f \in C([0,1],\R)$. We say $(\zeta,f)$ is a {\em tree-like path} 
\nomenclature[tree-like path]{Tree-like path}{A pair $(\zeta,f)$ with $\zeta$ an excursion, $f(0)=f(1)=1$, and $\Dist_\zeta(x,y)=0 \Rightarrow f(x)=f(y)$.}
if $\zeta$ is an excursion, $f(0)=f(1)=0$, and 
for all $x,y \in [0,1]$, if $\Dist_\zeta(x,y)=0$ then 
$f(x)=f(y)$. This implies that the push-forward of $f$ to $\cT_\zeta$ is well-defined and is continuous for $\Dist_\zeta$. 

It follows from the results in \cite{MarckertMokkademSnake} that if $(C,Z)$ is a distributional limit as in Theorem~\ref{thm:main} then $(C,Z)$ is a tree-like path, a fact which will be useful in the proof of the theorem. 

\subsection{Valid distributions and symmetrization}\label{sub:validAndSymmetric}
Let $(T,\rD)$ be a random labeled tree with type space $S$. Then $(T,\rD)$ satisfies (ii) and (iii) in the definition of validity if and only if there exists a set $\{\pi^r_{\rs}: r \in S, \rs \in \bigcup_{k \ge 1} S^k\}$, with each  $\pi^r_{(s_1,\ldots,s_k)}$ a probability measure on $\R^k$, such that the following holds. For any plane tree $t$ and any sets $(B_v,v \in V(t))$ with $B_v$ a Borel subset of $\R^{k_t(v)}$ for all $v$, 
\begin{equation}\label{eq:validity_rewrite}
\p{T=t, \forall v \in V(t), D_v \in B_v} = \p{T=t} \cdot \prod_{v \in V(t)} \pi^{s(v)}_{\cht(v)}(B_v)\, .
\end{equation}
In other words, writing $\pi_T$ for the conditional law of $\rD$ given $T$, the identity (\ref{eq:validity_rewrite}) states that
$\pi_T = \otimes_{v \in V(T)} \pi_{\cht(v)}^{s(v)}$. 
\medskip

Next, it is clear that if $\nu$ is valid 
then $\nu^\sym$ is also valid. 
 Moreover, the definition of symmetrization implies that if the displacement laws under $\nu$ are described by the measures 
$\{\pi^u_{\rs}: u \in S, \rs \in \bigcup_{k \ge 1} S^k\}$, then
 the displacement laws under $\nu^{\sym}$ are described by the measures 
$\{\pi^{u,\sym}_{\rs}: u \in S, \rs \in \bigcup_{k \ge 1} S^k\}$, where for any type $u$, any $k \ge 1$,  any type vector $\rs \in S^k$ and any Borel $B \subset \R^k$, we have 
\begin{equation}\label{def:piSym}
\pi^{u,\sym}_{\rs} (B) = \frac{1}{k!}\sum_{\sigma \in \mathfrak{S}_k} \pi^{u}_{\sigma(\rs)}(\sigma(B)).
\end{equation}

\subsection{Valid distributions and labeled Galton-Watson trees}\label{sub:validAndGW}
A multitype Galton-Watson tree with type space $S$ is defined by a collection $\mathrm{p}=(p^s,s \in S)$ with each $p^s$ a probability distribution on $S^\mathrm{fin}$, such that if $\vec{v},\vec{w} \in S^\mathrm{fin}$ are such that $\vec{w}$ is a permutation of $\vec{v}$ then $p^s(\vec{v})=p^s(\vec{w})$ for all $s \in S$. Observe that, writing $n_x(\vec{v})$ for the number of entries of $\vec{v}$ equal to $x \in S$, this means each $p^s$ is uniquely determined by the values 
\[
p^s(\{\vec{v} \in S^\mathrm{fin}: \forall x \in S, n_x(\vec{v})=i_x\}), 
\]
as $(i_x,x \in S)$ ranges over collections of non-negative integers with finite sum. 

A random tree $\cT$ is Galton-Watson$(\mathrm{p})$-distributed if the following holds. For all $n \ge 1$ and all rooted plane trees $t$ with node types in $S$ and with height at most $n$, writing $s$ for the type of the root of $t$, 
\[
\p{\cT_{\le n} = t~|~s(\emptyset)=s} = \prod_{v \in V(t_{<n})} p^{s(v)}(\cht_t(v))\, .
\]
Here we have written $\cT_{\le n}$ for the subtree of $\cT$ consisting of nodes at distance at most $n$ from the root (and likewise defined $t_{<n}$); also, we use the convention that $\cht_t(v)=()$ if $v$ has no children. 
In particular, this implies that, for $t$ a finite tree,
\[
\p{\cT = t~|~s(\emptyset)=s} = \prod_{v \in V(t)} p^{s(v)}(\cht_t(v))\, .
\]
From the assumption that the offspring distributions $(p^s,s \in S)$ are permutation-invariant, the following result is immediate.
\begin{prop}\label{prop:GWSym}
Let $\cT$ be a multitype Galton-Watson tree with type space $S$. Fix any type $s \in S$ and integer $n \ge 1$ such that $\p{s(\emptyset)=s,|\cT|=n}>0$. Then the conditional probability measure $\probC{\cdot}{s(\emptyset)=s,|\cT|=n}$ is symmetric. 
\end{prop}
\begin{remark}
We have built symmetry into our definition of multitype Galton-Watson trees. This is relatively standard (for example, it is also the case in~\cite{MiermontInvarianceTrees}), possibly because from the perspective of the (multitype) generation size process, there is no loss of generality in restricting to the symmetric case. One could of course study families of offspring distributions $\mathrm{p}=(p^s,s \in S)$  which are not assumed to be permutation-invariant; but this is beyond the scope of the current work.
\end{remark}

We next consider how the definition of valid laws interacts with that of multitype Galton-Watson trees.
Suppose that $(\cT,D)$ is a random labeled plane tree whose law is valid, and suppose there is $\mathrm{p}$ such that the underlying plane tree $\cT$ is Galton-Watson$(\mathrm{p})$-distributed.
Then for any finite plane tree $t$, writing $s$ for the type of the root of $t$, we have 
\[
\p{\cT=t} = \p{s(\emptyset)=s}\cdot \prod_{v \in V(t)} p^{s(v)}(\cht(v)), 
\]
and by (\ref{eq:validity_rewrite}), for any Borel sets $(B_v,v \in V(t))$, we then have
\begin{align*}
 \p{T=t, \forall v \in V(t), D_v \in B_v} 
& = \p{s(\emptyset)=s}\cdot \prod_{v \in V(t)} p^{s(v)}(\cht(v))\cdot\pi^{s(v)}_{\cht(v)}(B_v)\\
\, ,
\end{align*}
for an appropriate family $\{\pi^u_{\rs}: u \in S, \rs \in \bigcup_{k \ge 1} S^k\}$ of displacement laws. 
Moreover, for any $n \in \N$, we have 
\begin{align*}
&  \probC{\cT=t, \forall v \in V(t), D_v \in B_v}{|\cT|=n} \\
& = \frac{1}{\p{|\cT|=n}}\p{s(\emptyset)=s}\cdot \prod_{v \in V(t)} p^{s(v)}(\cht(v))\cdot\pi^{s(v)}_{\cht(v)}(B_v) \\
& = \probC{\cT=t}{|\cT|=n} 
\cdot \prod_{v \in V(t)} \pi^{s(v)}_{\cht(v)}(B_v)\, ,
\end{align*}
so the conditional law of $(\cT,D)$ given that $|\cT|=n$ is again valid. A similar logic shows that the conditional law of $(\cT,D)$ conditional on the type of its root, and on containing a fixed number of vertices of a given type, is also valid; we mention this as such a conditioning will arise later in the paper. 

\subsection{Locally centered, centered and globally centered displacements}\label{sub:centering}
Let $(T,\rD)$ be a tree sampled from a valid distribution. We consider $\bm{\pi}=\{\pi^r_{\rs}: r \in S, \rs \in \bigcup_{k \ge 1} S^k\}$, the family of the distributions of the vector of displacements. Recall that for $\vec{v} =(v_1,\ldots,v_k)\in S^{\mathrm{fin}}$ and $x \in s$ we write $n_x(\vec{v})=|\{1 \le i \le k:v_i=x\}|$. 
\begin{definition}
For each $\pi^r_{\rs}\in \bm{\pi}$, let $(X^r_{\rs,i},1\leq i \leq |\rs|)$ have law $\pi^r_\rs$. We say $\bm{\pi}$ is \emph{locally centered} if $\E{X^r_{\rs,i}}=0$ for all $r\in S$, $\rs \in S^{\mathrm{fin}}$ and $i\in \{1,\ldots,|\rs|\}$.

The family $\bm{\pi}$ is \emph{centered} if for any $k\in \mathbb{Z}_+$, for all $\rz=(z_s,s \in S)\in \mathbb{Z}_+^{S}$ with $\sum_{s \in S} z_s=k$ and for all $r\in S$, 
\[
\sum_{\{\rs\in S^k:(n_x(\rs),x\in S)=\rz\}} \sum_{i=1}^k\E{X^r_{\rs,i}}=0.
\]

\end{definition}
The assumption that $\bm{\pi}$ is locally centered appears often in the literature about the convergence of labeled Galton-Watson trees. As already alluded to in the introduction, the families of trees we want to study are not locally centered -- however, they are centered. The next claim, which is immediate from~\eqref{def:piSym}, says symmetrization turns centered displacements into locally centered displacements.
\begin{claim}
If $(T,\rD)$ is a tree sampled from a valid distribution with centered displacements, then the family of displacement distributions of $(T^{\sym},\rD^{\sym})$ is locally centered.
\end{claim}

For Galton-Watson trees, a number of asymptotic results have been obtained for displacements which are not locally centered but satisfy weaker centering assumptions. For example, in~\cite{MarckertLineage}, Marckert studies the convergence to the Brownian snake for labelled single type Galton-Watson trees, where the offspring distribution $\zeta$ is assumed to have bounded support and where the displacements are assumed to be {\em globally} centered, in that 
\[
	\sum_{k\geq 0}\zeta(k)\sum_{i=1}^k\E{X_{k,i}}=0,
\]
where $(X_{k,1},\ldots,X_{k,k})$ has the distribution of vector displacements to the children for a node with $k$ children.
We could not find a way to use symmetrization to transform globally centered into locally centered displacements. It is an open problem to know whether the bounded support assumption in Marckert's work can be relaxed.

\subsection{Subsampling in labeled trees}
A important step of the proof of Theorem~\ref{thm:main} is accomplished by the following lemma, which relates the laws of subtrees obtained by sampling in random labeled trees and their symmetrizations. Informally, the lemma states that the distribution of the subtree spanned by a set of randomly sampled vertices is the same in a random labeled tree and in its symmetrization, provided that all label displacements on child edges incident to branchpoints of the subsampled trees are ignored. 

For a plane tree $t$ and a sequence $\rv=(v_1,\ldots,v_k) \in V(t)^k$,  
write $t(\rv)$ 
\nomenclature[tv]{$t(\rv)$}{The subtree of $t$ spanned by vertices of $\rv$ and their ancestors.}
for the subtree of $t$ spanned by the vertices of $\rv$ together with all their ancestors in $t$. 
We view $t(\rv)$ as a plane tree by using the plane structure of $t$. The Ulam-Harris labels in $t(\rv)$ need not agree with those in $t$, so for a vertex $v$ which is an ancestor
of a vertex in $\rv$ (not necessarily strict), we write $U(\rv,v)$
\nomenclature[Uvv]{$U(\rv,v_i)$}{Ulam-Harris label of node of $t(\rv)$ corresponding to $v_i$.}
 for (the Ulam-Harris label of) the node corresponding to $v$ in $t(\rv)$. We also let $U(\rv)=(U(\rv,v_1),\ldots,U(\rv,v_k))$; this vector plays a key role in the coming lemma.

If $\rd \in \R^{e(t)}$ is a labeling of $t$ then we let $\rd(\rv)$ 
\nomenclature[dv]{$\rd(\rv)$}{Pushforward of $\rd$ to $\rt(\rv)$.}
be the pushforward of $\rd$ to $t(\rv)$; so if $e=uu' \in e(t(\rv))$ with $u=U(\rv,v),u'=U(\rv,v')$ then 
$\rd(\rv)_{uu'}=\rd_{vv'}$. 
We also define a modified labeling $\rd\angles{s}$
\nomenclature[dv]{$\rd\angles{\rv}$}{Modification of $\rd(\rv)$ so child displacements at branchpoints are zero.}
 as follows. For an edge $uu' \in e(t(\rv))$ with $u$ the parent of $u'$, let 
\begin{equation}\label{eq:defdbracket}
\rd\angles{\rv}_{uu'} = 
\begin{cases}
\rd(\rv)_{uu'} & \mbox{ if } k_{t(\rv)}(u) =1 \\
0		& \mbox{ otherwise.} 
\end{cases}
\end{equation}
Think of $\rd\angles{\rv}$ as ``ignoring displacements at branchpoints of $t(\rv)$''. 

Let $\rT=(T,D)$ be a random labeled plane tree. We say that a random vector $\rR=(R_1,\ldots,R_k)$ is uniformly sampled from $\rT$ if 
for all plane trees $t$, Borel sets $(B_e,e \in e(t))$ and vectors $\rr=(r_1,\ldots,r_k) \in V(t)^k$, 
\[
\p{T=t,\rR=\rr, D_e \in B_e~\forall e \in e(t)} = \frac{1}{|V(t)|^k} \p{T=t, D_e \in B_e~\forall e \in e(t)}\, .
\]

\begin{lem}\label{lem:eqdist}
Let $\rT=(T,\rD)$ be a random labeled multitype plane tree with valid law $\nu$, and let $\rT^\sym=(T^\sym,\rD^\sym)$ have law $\nu^\sym$. 

Fix $k\in \N$, and let $\rR=(R_1,\ldots,R_k)$ and $\rQ=(Q_1,\ldots,Q_k)$ be random vectors of length $k$, uniformly sampled from $\rT$ and $\rT^\sym$ respectively. 
Then $(T(\rR),\rD\angles{\rR},U(\rR))$
and $(T^\sym(\rQ),\rD^\sym\angles{\rQ}, U(\rQ))$ are equal in distribution. 
\end{lem}

\begin{proof}\label{proof:Lemma}
Given a tree $t$ and a sequence 
$\ru=(u_1,\ldots,u_k) \in V(t)^k$, let $\branch(t,\ru)$ 
\nomenclature[branchtu]{$\branch(t,\ru)$}{Vertices of $t$ which correspond to  branchpoints of $t(\ru)$.}
be the set of vertices of $t$ possessing at least two distinct children with descendants in $\ru$ (in other words, these are the vertices of $t$ which correspond to branchpoints of $t(\ru)$). Further, let $\Permvecs_{(t,\ru)}$ 
\nomenclature[Ptu]{$\Permvecs_{(t,\ru)}$}{Permutations $\sigma \in \Permvecs_t$ leaving ordering unchanged at $v \in \branch(t,\ru)$.}
be the set of vectors $\sigma=(\sigma_v,v\in t)\in \Permvecs_t$ such that $\sigma_v$ is the identity permutation for all $v\in \branch(t,\ru)$.

Now let $(T,\rD)$ have law $\nu$, let $\rR=(R_1,\ldots,R_k) \in V(T)^k$ be a random vector of length $k$ uniformly sampled from $T$, and let $\sigma \in_u \Permvecs_{(T,\rR)}$. 
We construct another labelled tree $(\hat{T},\hat{\rD})$ from $(T,\rD)$ as follows. In words we use $\sigma$ to perform a full symmetrization at all vertices of $T$ except at vertices corresponding to branchpoints of $T(\rR)$; at the latter vertices we don't permute the children and we set all the displacements to zero.  
Formally, set $\hat{T} = \sigma(T)$, for $1\le i \le k$ set $\hat R_i = \sigma(R_i)$, and let $\hat\rR = (\hat{R}_1,\ldots,\hat{R}_k)$. Then
for $(v,vi) \in e(T)$ let 
\[ 
\hat D_{\sigma(v),\sigma(vi)}=
\begin{cases}
0& \mbox{if } v \in \branch(T,\rR)\\
D_{v,vi} & \mbox{otherwise},
\end{cases}
\]
so if $e=vw \in E(\hat{T})$ then $\hat{D}_e = D_{\sigma^{-1}(v)\sigma^{-1}(w)}$. 
Now set $\hat\rD=(\hat{D}_e,e\in e(\hat T))$. 

\begin{figure}[h]
  \begin{center}
    \subfloat[Symmetrization outside the branchpoints.]{
      \includegraphics[width=0.4\textwidth,page=1]{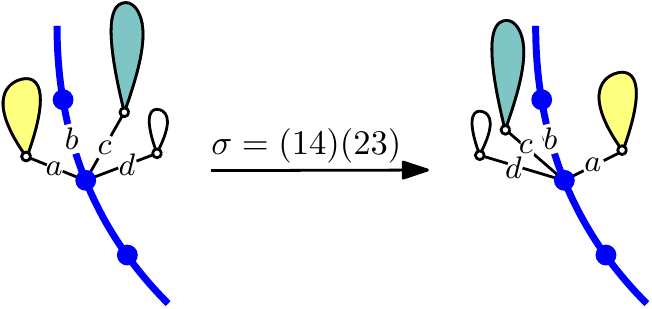}
      \label{figsub:symnormal}
                         }\qquad \qquad \qquad               
    \subfloat[Symmetrization at a branch point.]{
      \includegraphics[width=0.4\textwidth,page=2]{Figures/symmetrization.pdf}
      \label{figsub:symbranch}
                         }
    \caption{Examples of symmetrization. In both figures, the tree $\rt$ is on the left and $\hat \rt$ on the right. Labels $a$, $b$, $c$ and $d$ represent the displacements along the edges. Branches of $t(\ru)$ are represented in bold blue.}
    \label{fig:symmetrization}
  \end{center}
\end{figure}

Here is an important property of the preceding construction. For all $1\le i,j\le k$, if $R_i \prec R_j$ (where $\prec$ corresponds to the lexicographic ordering on the Ulam-Harris encoding) then  $\hat R_i \prec \hat R_j$. This immediately implies that $(T(\rR),U(\rR)) = (\hat T (\hat R),U(\hat \rR))$. Moreover, the only differences between $(T(\rR),\rD(\rR))$ and $(\hat T(\hat R),\hat{\rD}(\hat R))$ occur at nodes with at least 
two children in $T(\rR)$.
Since the displacements on edges leaving such nodes are set to 0 when passing from $\rD(\rR)$ to $\rD\angles{\rR}$  (see \eqref{eq:defdbracket}), it follows that $\hat \rD\angles{\hat \rR} = \rD\angles{\rR}$ as well, so $(\hat T(\hat \rR),\hat\rD\angles{\hat\rR},U(\hat \rR)) = (T(\rR),\rD\angles{\rR},U(\rR))$. 

For $v \in V(\hat{T})$ we write $\hat{D}_v = (\hat{D}_{v,vi},1 \le i \le k(v))$. 
Now fix a tree $t$ and a length-$k$ vector $\rr=(r_1,\ldots,r_k)$ of vertices of $t$. We will show that for any collection $(B_v, v\in V(t)\backslash \branch(t,\rr))$, with each $B_v$ a Borel set of $\R^{k_t(v)}$, 
\begin{align*}
&\p{(\hat T,\hat \rR)=(t,\rr) \text{ and } \hat D_v \in B_v\ \forall v \in V(t)\backslash \branch(t,\rr)}\\ = &\p{(T^\sym,\rQ)=(t,\rr) \text{ and }D^\sym_v \in B_v\ \forall v \in V(t)\backslash \branch(t,\rr)}.
\end{align*}
This equality implies that $(\hat T(\hat \rR),\hat\rD\angles{\hat\rR},U(\hat \rR))$ and $(T^\sym(\rQ),\rD^\sym\angles{\rQ},U(\rQ))$ are equal in distribution. Since $(\hat T(\hat \rR),\hat\rD\angles{\hat\rR},U(\hat \rR)) = (T(\rR),\rD\angles{\rR},U(\rR))$, this will prove the lemma for single type trees. 

For $\tau \in \Permvecs_{(t,\rr)}$, we let $\tau(t,\rr)=(\tau(t),(\tau(r_1),\ldots,\tau(r_k))$. 
For $\tau \in \Permvecs_t$, we let $\tau^{*}$ be the element of $\Permvecs_{\tau(t)}$ defined by setting $\tau^*_{\tau(v)} = (\tau_v)^{-1}$ for all $v \in t$. Note that $\tau^*(\tau(t))=t$ -- so $\tau^*$ acts as an inverse to $\tau$ -- and that if $\tau \in \Permvecs_{(t,\rr)}$ then $\tau^*\in \Permvecs_{\tau(t,r)}$.

Since $\hat T=\sigma(T)$ we have  
\begin{align*}
& \p{(\hat T,\hat \rR)=(t,\rr)\text{ and } \hat D_v \in B_v\ \forall v \in V(t)\backslash \branch(t,\rr)}\\
&=\sum_{\tau \in \Permvecs_{(t,\rr)}} 
\p{\tau^*(T,\rR)=(t,\rr),\ \sigma=\tau^*\text{ and } \hat D_v \in B_v\ \forall v \in V(t)\backslash \branch(t,\rr)}\\
&=\sum_{\tau \in \Permvecs_{(t,\rr)}}
\p{(T,\rR)=\tau(t,\rr),\ \sigma=\tau^* \text{ and } 
D_{\tau(v)} \in \tau_v(B_v)\ \forall v \in V(t)\backslash \branch(t,\rr)},
\end{align*}

Now 
note that 
 \[
 |\Permvecs_{(t,\rr)}|= \prod_{v\in V(t)\backslash \branch(t,\rr)} k_t(v)!=|\Permvecs_{\tau(t,\rr)}|,\]
 as $\branch{(\tau(t,\rr))} = \{\tau(v),\, v\in \branch(t,\rr)\}$.
Using that $T$ is symmetric, that the elements of $\rR$ are uniformly sampled from $T$, and that $\sigma \in_u \Permvecs_{(T,\rR)}$, it follows that for all $\tau \in \Permvecs_t$, 
\begin{equation}\label{eq:finmono}
\p{(T,\rR)=\tau(t,\rr),\ \sigma=\tau^*}
= 
\frac{\p{T=t}}{|t|^k}\cdot \frac{1}{|\Permvecs_{\tau(t,\rr)}|}\,
= 
\frac{\p{T=t}}{|t|^k}\cdot \frac{1}{|\Permvecs_{(t,\rr)}|}\, . 
\end{equation}
Now note that for any $v\in V(t)$ and any $\tau \in \Permvecs_{(t,\rr)}$, $s(\tau(v))=s(v)$ and $\cht(\tau(v)) = \tau_v(\cht(v))$.
Together with (\ref{eq:validity_rewrite}) and (\ref{eq:finmono}), this implies that 
\begin{align*}
 & \p{(\hat T,\hat \rR)=(t,\rr)\text{ and } \hat D_v \in B_v\ \forall v \in V(t)\backslash \branch(t,\rr)}\\
 & = 
 \frac{\p{T=t}}{|t|^k} \cdot \sum_{\tau \in \Permvecs_{(t,\rr)}} 
 \frac{1}{|\Permvecs_{(t,\rr)}|}\, 
 \prod_{v \in V(t)\backslash \branch(t,\rr)} \pi^{s(\tau(v))}_{\cht(\tau(v))}(\tau_v(B_v)) 
\\
& = 
 \frac{\p{T=t}}{|t|^k} \cdot \sum_{\tau \in \Permvecs_{(t,\rr)}} 
 \frac{1}{|\Permvecs_{(t,\rr)}|}\, 
 \prod_{v \in V(t)\backslash \branch(t,\rr)} \pi^{s(v)}_{\tau_v(\cht(v))}(\tau_v(B_v)) 
 \end{align*}
From this, applying the definition of $\pi^{u,\sym}_{\rs}$ given in \eqref{def:piSym} then yields that 
\begin{align*}
 & \p{(\hat T,\hat \rR)=(t,\rr)\text{ and } \hat D_v \in B_v\ \forall v \in V(t)\backslash \branch(t,\rr)}\\
& = 
\frac{\p{T=t}}{|t|^k}
\cdot \prod_{v \in V(t)\backslash \branch(t,\rr)} \pi_{\cht(v)}^{s(v),\sym}(B_v)\\
& = 
\p{(T^\sym,\rQ)=(t,\rr) \text{ and }D^\sym_v \in B_v\ \forall v \in t \backslash \branch(t,\rr)},
\end{align*}
as required, which completes the proof of the lemma. 
\end{proof}
The key point in the preceding argument is that, because we ignore displacements at branchpoints, the probability factorizes due to (\ref{eq:validity_rewrite}).

\section{Proof of Theorem~\ref{thm:main}}\label{sec:proofMain}
In this section we explain how Theorem~\ref{thm:main} follows from Lemma~\ref{lem:eqdist}. For the remainder of the section, fix labeled trees $\rT_n=(T_n,\rD_n)$ and $\rT_n^\sym=(T_n^\sym,\rD_n^\sym)$ for $n \ge 1$, satisfying the conditions in Theorem~\ref{thm:main}, and suppose that there exists a random $C([0,1],\R^2)$-valued random proces $(C,Z)$ and positive sequences $(a_n,n \ge 1)$ and $(b_n,n \ge 1)$ with $a_n \to 0$ and $b_n \to 0$ such that 
\[
(a_nC_{\rT_n^{\mathrm{sym}}},b_nZ_{\rT_n^{\mathrm{sym}}}) \convdist (C,Z)\]
 for the uniform topology. We must show that the 
 same distributional limit holds for $(a_nC_{\rT_n},b_nZ_{\rT_n})$. 
 In the coming arguments we write $\theta_n=\theta_{T_n}$, $C_n=C_{T_n}$, $Z_n=Z_{T_n}$ and $\ell_n=\ell_{T_n}$to simplify notation, and similarly write $\Dist_n^\sym = \Dist_{C_{T_n^\sym}}$ et cetera. 

In brief, the proof proceeds as follows. To prove convergence in distribution it suffices to prove convergence of finite-dimensional distributions (FDDs), plus tightness. Lemma~\ref{lem:eqdist} will yield convergence of random FDDs; by sampling sufficiently many random points we may use this to show convergence of arbitrary FDDs. Tightness will follow fairly easily from the convergence for the symmetrized process and fact that, aside from the plane structure, a labeled tree is identical to its symmetrization.

In the proof of the next lemma we use the following definition. 
Fix a plane tree $t$ and write $n=2|V(t)|-2$. For $0 \le y <1$, let $v(t,y)$ 
be whichever of $\theta_t(\lfloor n y \rfloor)$ and $\theta_t(\lfloor n y \rfloor+1)$ is further from the root. Note that if $U$ is uniformly distributed on $[0,1]$ then $v(t,U)$ is a uniformly random non-root node of $t$.  
\begin{lem}\label{lem:fdd}
Let $(\unif_i,i \ge 1)$ be independent Uniform$[0,1]$ random variables, independent of the trees $\rT_n$. 
Fix $k \ge 1$ and write $(\unif_i^\uparrow,1 \le i \le k)$ for the increasing reordering of $\unif_1,\ldots,\unif_k$. Then 
\[
(a_n C_{n}(\unif_i^\uparrow), b_n Z_{n}(\unif_i^\uparrow),i \le k) \convdist (C(\unif_i^\uparrow),Z(\unif_i^\uparrow),i \le k)\, 
\]
as $n \to \infty$. 
\end{lem}
\begin{proof}
We note at the outset that, since the entries of $(X_i,i \ge 1)$ are independent of the trees $\rT_n$, they are also independent of $C_n$ and of $Z_n$, since these two random functions are measurable with respect to $\rT_n$.

Let $\rR=(R_1,\ldots,R_k)$ and $\rQ=(Q_1,\ldots,Q_k)$ be vectors of independent uniform samples from $\rT_n$ and $\rT_n^\sym$ as in Lemma~\ref{lem:eqdist}. (We leave the dependence of $\rR$ and $\rQ$ on $n$ implicit.)
The conclusion of that lemma implies that $\rT_n(R)$ conditioned on the event that $\emptyset \notin R$ has same distribution as $\rT_n^\sym$ conditioned on the event that $\emptyset \notin Q$.

For each $1 \le i \le k$, let $R_i'=v(T_n,\unif_i)$ and let $Q_i' = v(T_n^\sym,\unif_i)$; then let $\rR'=(R_i',i \le k)$ and let $\rQ'=(Q_i',i \le k)$. 
Let $E$ be the event that $\emptyset \not \in \{R_1,\ldots,R_k\}$, so the conditional law of 
$(\rT_n( \rR),U(\rR))$ given $E$ is the law of 
$(\rT_n( \rR'),U(\rR'))$. Similarly, the conditional law of 
$(\rT_n^\sym( \rQ),U(\rQ))$ given $\emptyset \not \in \{Q_1,\ldots,Q_k\}$ is the law of 
$(\rT^\sym_n( \rQ'),U(\rQ'))$.
 By Lemma~\ref{lem:eqdist}, it then follows that 
$(T_n(\rR'),\rD\angles{\rR'},U(\rR'))$ and $(T_n^\sym(\rQ'),\rD^\sym\angles{\rQ'},U(\rQ'))$ are also identically distributed. 



Using the conclusion of the preceding paragraph, we may fix a coupling which makes 
\[
(T_n(\rR'),\rD\angles{\rR'},U(\rR'))\aseq (T_n^\sym(\rQ'),\rD^\sym\angles{\rQ'},U(\rQ'))\, .
\]
Letting $R_i^\uparrow = v(T_n,\unif_i^\uparrow)$, $Q_i^\uparrow = v(T_n^\sym,\unif_i^\uparrow)$ and 
$\rR^\uparrow = (R_i^\uparrow,i \le k)$, $\rQ^\uparrow = (Q_i^\uparrow,i \le k)$, we then have 
\begin{equation}\label{eq:as_qr_identity}
(T_n(\rR^{\uparrow}),\rD\angles{\rR^{\uparrow}},U(\rR^{\uparrow}))\aseq (T_n^\sym(\rQ^{\uparrow}),\rD^\sym\angles{\rQ^{\uparrow}},U(\rQ^{\uparrow}))\, .
\end{equation}

Now write $\Delta_n$ and $\Delta_n^\sym$ for the greatest absolute values of an  edge label of $\rT_n$ and $\rT_n^\sym$, respectively, i.e. 
\[
\Delta_n = 
\sup_{uv \in e(T_n)} |\ell_n(u)-\ell_n(v)| \qquad \mbox{and} \qquad
\Delta_n^\sym = \sup_{uv \in e(T_n^\sym)} |\ell_n^\sym(u)-\ell_n^\sym(v)|\, .
\]
For all $i \le k$, the difference between $\ell_n(R_i^\uparrow)$ and the label of $U(\rR^\uparrow,R_i^\uparrow)$ in $(T_n(\rR^\uparrow),\rD\angles{\rR^\uparrow})$ is at most $(k-1) \Delta_n$, since any difference between these labels is caused exclusively by the zeroing of labels in $\rD\angles{\rR^\uparrow}$ at branchpoints, and there are at most $k-1$ branchpoints of $T_n(\rR^\uparrow)$ along any path from the root. Likewise, the difference between 
$\ell_n^\sym(Q_i^\uparrow)$ and the label of $U(\rQ^\uparrow,Q_i^\uparrow)$ in $(T_n^\sym(\rQ^\uparrow),\rD\angles{\rR^\uparrow})$ is at most $(k-1) \Delta_n^\sym$. 
It then follows from (\ref{eq:as_qr_identity}) that
\[
|\ell_n^\sym(Q_i^\uparrow)-\ell_n(R_i^\uparrow)|
\le (k-1)(\Delta_n+\Delta_n^\sym)\, .
\]
Also, the value of $Z_n(\unif_i^\uparrow)$ lies between $\ell_n(R_i^\uparrow)$ and the label of one of its neighbours in $\rT_n$, and the 
value $Z_n^\sym(\unif_i^\uparrow)$ lies between $\ell_n^\sym(Q_i^\uparrow)$ and the label of one of its neighbours in $\rT_n^\sym$, so
\[
|Z_n(\unif_i^\uparrow) - \ell_n(R_i^\uparrow)| \le \Delta_n
\mbox{ and } 
|Z_n^\sym(\unif_i^\uparrow) - \ell_n^\sym(Q_i^\uparrow)| 
\le \Delta_n^\sym\, . 
\]
It follows that 
\begin{align}\label{eq:z_diff_delta_bd}
b_n \sup_{i \le k} |Z_n^\sym(\unif_i^\uparrow) - Z_n(\unif_i^\uparrow)| 
& \le k  b_n (\Delta_n+\Delta_n^\sym)
\end{align}

Now notice that, writing $N=2|T_n^\sym|-2$, we may represent $\Delta_n$ and $\Delta_n^\sym$ as 
\[
\Delta_n 
= \sup_{|x-y| \le 1/N} |Z_n(x)-Z_n(y)|
\qquad \mbox{and} \qquad
\Delta_n^\sym 
= \sup_{|x-y| \le 1/N} |Z_n^\sym(x)-Z_n^\sym(y)|\, .
\]
Since $(C,Z)$ is a $C([0,1],\R^2)$-valued process, $Z$ itself is a $C([0,1],\R)$-valued process, so is almost surely uniformly continuous. Since $(a_nC_n^\sym,b_nZ_n^\sym) \convdist (C,Z)$ it follows that $b_n Z_n^\sym \convdist Z$ for the uniform topology on $C([0,1],\R)$. The second of the preceding equalities then implies that $b_n \Delta_n^\sym \convdist 0$. 

For any labelled tree $\rt=(t,\rd)$ and any $\sigma \in P_t$, the multiset of edge labels is the same in $\rt$ and in $\sigma(\rt)$, so in particular the  largest absolute value of an edge label is the same in both trees. It thus follows from the definition of symmetrization that $\Delta_n \eqdist \Delta_n^\sym$.
Since $b_n\Delta_n^\sym \convdist 0$ it follows that $b_n \Delta_n \convdist 0$ as well, and (\ref{eq:z_diff_delta_bd}) then implies that 
\[
b_n \sup_{i \le k} |Z_n^\sym(\unif_i^\uparrow) - Z_n(\unif_i^\uparrow)| 
\convdist 0\, .
\]

Under the coupling which yields (\ref{eq:as_qr_identity}), we also have 
\[
(C_n(\unif_i^\uparrow),i \le k)=(C_n^\sym(\unif_i^\uparrow),i \le k)\, ,
\]
from which it follows that 
$(a_n C_n(\unif_i^\uparrow),b_n Z_n(\unif_i^\uparrow))$
and 
$(a_n C_n^\sym(\unif_i^\uparrow),b_n Z_n^\sym(\unif_i^\uparrow))$
must have the same distributional limit. By the assumption of Theorem~\ref{thm:main}, 
the latter converges to $(C(\unif_i^\uparrow),Z(\unif_i^\uparrow),i \le k)$. 
\end{proof}

\begin{lem}\label{lem:boundLabels}
Let $\rT=(T,\rD)$ be a random labeled tree and let 
$\rT^\sym$ be its symmetrization. 
Then for any constants $k,K,M$, we have 
\begin{align*}
& \p{\sup_{|i-j| \le k} |\ell_\rT(\theta_T(i)) -\ell_\rT(\theta_T(j))| > M}\\
\le & 
\p{\sup_{|i-j| \le k} \dist(\theta_T(i),\theta_T(j)) > K}
 + 
\p{
\sup_{v,w \in \rT^\sym: \dist(v,w) \le K} 
|\ell_{\rT^\sym}(v) - \ell_{\rT^\sym}(w)| > M}\, .
\end{align*}

\end{lem}
\begin{proof}
For vertices $v,w$ of a rooted plane tree, write $\bbr{v,w}$ 
\nomenclature[bracketvw]{$\bbr{v,w}$}{The unique path between vertices $v$ and $w$ belonging to some plane tree.}
for the unique path between $v$ and $w$; this path is determined by the Ulam-Harris  labels of $v$ and $w$ themselves, so there is no need to indicate the tree to which $v$ and $w$ belong in the notation.
Now fix a labeled tree $\rt=(t,\rd)$ and let $\sigma \in \Permvecs_\rt$. Note that for any $v,w \in V(t)$, the paths 
$\bbr{v,w}$ and $\bbr{\sigma(v),\sigma(w)}$ are identical, in that they have the same length, and visit edges with the same labels, in the same order. 
It follows that for all $i,j \le 2|V(t)|-2$, 
\begin{align*}
\dist(\theta_t(i),\theta_t(j)) & = \dist(\sigma(\theta_t(i)),\sigma(\theta_t(j))) \mbox{ and } \\
\ell_\rt(\theta_t(i)) & = \ell_{\sigma(\rt)}(\sigma(\theta_t(i))) \, .
\end{align*}
The second identity implies that for any $k$, 
\[
\sup_{|i-j| \le k} |\ell_\rt(\theta_t(i)) -\ell_\rt(\theta_t(j))|
=
\sup_{|i-j| \le k} |\ell_{\sigma(\rt)}(\sigma(\theta_t(i))) -
\ell_{\sigma(\rt)}(\sigma(\theta_t(j)))|. 
\]
The first identity implies that 
\[
\sup_{|i-j| \le k} \dist(\theta_t(i),\theta_t(j))
=
\sup_{|i-j| \le k} \dist(\sigma(\theta_t(i)),\sigma(\theta_t(j)))\, .
\]
For any constants $M,K$, it follows that if 
\[
\sup_{|i-j| \le k} |\ell_\rt(\theta_t(i)) -\ell_\rt(\theta_t(j))| > M
\]
then either 
\[
\sup_{|i-j| \le k} \dist(\theta_t(i),\theta_t(j)) > K
\]
or 
\[
\sup_{v,w \in V(\sigma(t)): \dist(v,w) \le K} 
|\ell_{\sigma(\rt)}(v) - \ell_{\sigma(\rt)}(w)| > M\, .
\]
We now apply this to the random tree $\rT=(T,\rD)$, 
and to $\sigma \in_u \Permvecs_{\rT}$. By a union bound, this gives 
\begin{align*}
& \p{\sup_{|i-j| \le k} |\ell_\rT(\theta_T(i)) -\ell_\rT(\theta_T(j))| > M}\\
\le & 
\p{\sup_{|i-j| \le k} \dist(\theta_T(i),\theta_T(j)) > K}
 + 
\p{
\sup_{v,w \in V(\sigma(T)): \dist(v,w) \le K} 
|\ell_{\sigma(\rT)}(v) - \ell_{\sigma(\rT)}(w)| > M}\, ,
\end{align*}
which concludes the proof of the lemma, since $\sigma(\rT) \eqdist \rT^\sym$.
\end{proof}
\begin{lem}\label{lem:contourTight}
Under the hypotheses of Theorem~\ref{thm:main}, for all $\beta>0$, there exists $\alpha = \alpha(\beta)>0$ such that 
\begin{equation*}
\limsup_n\p{\sup_{|i-j| \le \lfloor \alpha |T_n| \rfloor} a_n \dist\big(\theta_{n}(i),\theta_{n}(j)\big)> \beta}	<\beta.	
\end{equation*}	
\end{lem}
\begin{proof}
Fix $\beta>0$.
Since $\mu_n$ is symmetric, $T_n$ and $T^\sym_n$ have the same distribution (as unlabeled plane trees). Thus, the convergence result for the contour of $(T^\sym_n)$ translates immediately into the same result for the contour of $(T_n)$. This implies in particular that the process $(a_nC_n)$ is tight, so there exists $\alpha>0$ such that
\begin{equation}\label{eq:contourtight}
\limsup_n\p{\sup_{|x-y| \le \alpha } a_n |C_n(x)-C_n(y)|> \beta}	<\beta.
\end{equation}		
Now, for any $n$ and $0\le i\le j \le 2|T_n|-2$, observe that
\[
\dist\big(\theta_n(i),\theta_n(j)\big) = C_n\big(\frac{i}{2|T_n|-2}\big)+C_n\big(\frac{j}{2|T_n|-2}\big)-2 \inf_{i\le k\le j}C_n\big(\frac{k}{2|T_n|-2}\big),
\]
which together with \eqref{eq:contourtight} and the triangle inequality gives the desired result.
\end{proof}
\begin{lem}\label{lem:labelTight}
Under hypotheses of Theorem~\ref{thm:main}, for all $\epsilon>0$, there exists $\beta=\beta(\epsilon)>0$ such that 
\begin{equation*}
\sup_n\p{\sup_{v,w \in T_n^\sym:\dist(v,w) \le \beta/a_n} b_n 
|\ell^\sym_n(v) -\ell^\sym_n(w)|>\epsilon}	<\epsilon
\end{equation*}	

\end{lem}
\begin{proof}
As noted in the introduction, 
if $(\zeta,f)$ is a tree-like path then $f$ can be pushed forward to $\cT_\zeta$ and is continuous on that domain. Since $\cT_\zeta$ is compact, $f$ is in fact uniformly continuous on $\cT_\zeta$. 
In the current setting, this implies that the push-forward of $Z$ to $\cT_C$ is a.s.\ uniformly continuous on $\cT_C$ with respect to (the pushforward of) $\Dist_C$. Thus, for all $\epsilon>0$, there exists $\beta>0$ such that 
\begin{equation}\label{eq:snakeContinuous}
\p{\sup_{x,y \in [0,1]:\Dist_C(x,y) \le \beta} |Z(x)-Z(y)| \ge \eps} < \eps\, .
\end{equation}
Since $(a_nC_n^\sym,b_nZ_n^\sym) \convdist (C,Z)$ by assumption, after decreasing $\beta$ if necessary, \eqref{eq:snakeContinuous} implies that 
\[ 
\sup_n\p{\sup_{x,y: \Dist_n^\sym(x,y)\le \beta/a_n}\{b_n|Z_n^\sym(x)-Z_n^\sym(y)|\}>\eps}<\eps.
\]
Since $T_n^\sym$ is isometric to a subspace of $\cT_{C_n^\sym}$, we have 
\[
\sup_{v,w \in T_n^\sym:\dist(v,w) \le \beta/a_n} b_n 
|\ell^\sym_n(v) -\ell^\sym_n(w)|
\le \sup_{x,y: \Dist_n^\sym(x,y)\le \beta/a_n}\{b_n|Z_n^\sym(x)-Z_n^\sym(y)|\}, 
\]
and the result follows. 
\end{proof}

\begin{prop}\label{prop:tightness}
Under the hypotheses of Theorem~\ref{thm:main}, the family $(b_nZ_n)$ is tight.
\end{prop}
\begin{proof}
Fix $\epsilon>0$, let $\beta<\eps$ be such that the bound of Lemma~\ref{lem:labelTight} holds and let $\alpha = \alpha(\beta)$ be such that the bound of Lemma~\ref{lem:contourTight} holds. 
Let $\rt=\rT_n$ and $N=2|\rT_n|-2$; we assume 
$n$ is large enough that $\alpha N \ge 2$. Applying  Lemma~\ref{lem:boundLabels} with $k=\alpha N$, $K= \beta/a_n$ and $M=\epsilon/b_n$, we get 
\begin{align*}
& \p{\sup_{|i-j| \le \alpha N} |\ell_n(\theta_n(i)) -\ell_n(\theta_n(j))| > \epsilon/b_n}\\
\le & 
\p{\sup_{|i-j| \le\alpha N } \dist(\theta_n(i),\theta_n(j)) > \beta/a_n}
 + 
\p{
\sup_{\dist(v,w) \le \beta/a_n} 
|\ell_n^\sym(v) - \ell_n^\sym(w)| > \epsilon/b_n}\, .
\end{align*}
By Lemmas~\ref{lem:contourTight} and~\ref{lem:labelTight}, it follows that
\begin{equation}\label{eq:lnTight}
\limsup_n \p{\sup_{|i-j| \le \alpha N} |\ell_n(\theta_n(i)) -\ell_n(\theta_n(j))| > \epsilon/b_n} \le \beta + \epsilon \le 2\epsilon\, .
\end{equation}
For any $x,y \in [0,1]$, by the triangle inequality, we have 
\begin{multline*}
|Z_n(x)-Z_n(y)| \le |\ell_n(\theta_n(\lfloor xN\rfloor))-\ell_n(\theta_n(\lfloor yN\rfloor))|\\+|Z_n(x)-\ell_n(\theta_n(\lfloor xN\rfloor))|+|Z_n(y)-\ell_n(\theta_n(\lfloor yN\rfloor))|.
\end{multline*}
By definition of $Z_n$ this yields 
\begin{multline*}
|Z_n(x)-Z_n(y)| \le |\ell_n(\theta_n(\lfloor xN\rfloor))-\ell_n(\theta_n(\lfloor yN\rfloor))|\\+|\ell_n(\theta_n(\lceil xN\rceil))-\ell_n(\theta_n(\lfloor xN\rfloor))|+|\ell_n(\theta_n(\lceil yN\rceil))-\ell_n(\theta_n(\lfloor yN\rfloor))|.
\end{multline*}
Since $\alpha N \ge 2$, if $|x-y| \le \alpha/2$ 
then $|\lfloor x N \rfloor - \lfloor y N \rfloor| \le \alpha N$, so 
\[
\sup_{|x-y|\le \alpha/2}|Z_n(x)-Z_n(y)| \le 3 \sup_{|i-j| \le \lfloor \alpha N\rfloor} |\ell_n(\theta_n(i)) -\ell_n(\theta_n(j))|\, . 
\]
Together with Equation~\eqref{eq:lnTight}, this establishes the requisite tightness.
\end{proof}

\begin{proof}[Proof of Theorem~\ref{thm:main}]
For $n \ge 1$ write $\cL_n$ for the law of $(a_n C_n,b_nZ_n)$, so $\cL_n$ is a Borel probability measure on $C([0,1],\R)^2$. Proposition~\ref{prop:tightness} and (\ref{eq:contourtight}) together imply that the family $(\cL_n,n \ge 1)$ is tight.
To complete the proof, we establish convergence of 
finite-dimensional distributions by showing that for any $0 \le t_1 < t_2 < \ldots < t_m \le 1$ and any bounded Lipschitz function $F : \big(\mathbb{R}^2\big)^m \rightarrow \mathbb{R}$, 
\begin{equation}\label{eq:toprove}
\E{F\Big(\big(a_nC_n(t_i),b_nZ_n(t_i)\big)_{1\le i \le m}\Big)} \to \E{F\Big(\big(\limC(t_i),\limZ(t_i)\big)_{1\le i \le m}\Big)}
\end{equation}
For the remainder of the proof, fix $F$ and $(t_i, 1 \le i \le m)$ as above, let $\ninf{F}$ be the uniform norm of $F$, and let $\lip{F}$ be the Lipschitz constant of $F$ with respect to this norm. 

By \citep[Theorem 8.2]{billingsley}, since $(\cL_n,n \ge 1)$ is tight, for all $\delta > 0$ there is $\alpha=\alpha(\delta)$ such that 
\begin{equation} \label{eq:fin_tight}
\limsup_{n \to \infty} \p{\sup_{x,y \in [0,1],|x-y| \le \alpha} \left( a_n|C_n(x)-C_n(y)|+b_n|Z_n(x)-Z_n(y)|\right) > \delta} < \delta\, .
\end{equation} 
We write $(A_n,n \ge 1)$ for the events whose probabilities are bounded in (\ref{eq:fin_tight}).

Since $C$ and $Z$ are almost surely uniformly continuous, by decreasing $\alpha(\delta)$ if necessary we may additionally ensure that 
\begin{equation} \label{eq:lim_tight}
\p{\sup_{x,y \in [0,1],|x-y| \le \alpha} \left(|C(x)-C(y)| + |Z(x)-Z(y)|\right) > \delta} < \delta. 
\end{equation}
Let $(U_i,i \ge 1)$ be independent Uniform$[0,1]$ random variables, and for $k \ge 1$, let $(U_{i;k}^\uparrow,1 \le i \le k)$ be the increasing reordering of $U_1,\ldots,U_k$. Then the sequence of random 
variables $(\sup_{0 \le x \le 1} |U_{\lfloor kx\rfloor;k}^\uparrow-x|)_{k \ge 1}$ converges in probability to $0$ as $k \to \infty$. 
For $\delta > 0$, letting $\alpha=\alpha(\delta)$ be as above, we may therefore choose $j> 2/\alpha $ large enough that 
\[
\p{\max_{1 \le i \le j} \left| U_{i;j}^{\uparrow} - \frac{i}{j} \right| \ge \frac{\alpha}{2}} < \delta. 
\]
Now choose integers $k_1,\ldots,k_m$ so that for $1 \le i \le m$, $|k_i/j-t_i| < \alpha/2$; this is possible since $j > 2/\alpha$. It follows that 
\begin{equation}\label{eq:unif_llm}
\p{\max_{1 \le i \le m} |U^\uparrow_{k_i;j}-t_i|  \ge \alpha} < \delta\, .
\end{equation}

Let $B$ and $C$ be the events whose probabilities are bounded in (\ref{eq:lim_tight}) and (\ref{eq:unif_llm}). 
Writing $E_n=(A_n \cup B \cup C)^c$, we then have $\liminf_n \p{E_n} >1-3\delta$. When $E_n$ occurs, 
\begin{equation*}
\bigg|F\Big(\big(a_nC_{\rT_n}(t_i),b_nZ_{\rT_n}(t_i)\big)_{1\le i \le m}\Big)-
 F\Big(\big(a_nC_{\rT_n}(U^\uparrow_{k_i;j}),b_nZ_{\rT_n}(U^\uparrow_{k_i;j})\big)_{1\le i \le m}\Big)\bigg| \leq \delta \lip{F}\, ,
\end{equation*}
so for $n$ sufficiently large, 
\begin{align}\label{eq:cvgCZ}
&\bigg|\E{F\Big(\big(a_nC_n(t_i),b_nZ_n(t_i)\big)_{1\le i \le m}\Big)}-
 \E{F\Big(\big(a_nC_n(U^\uparrow_{k_i;j}),b_nZ_n(U^\uparrow_{k_i;j})\big)_{1\le i \le m}\Big)}\bigg| \nonumber\\ \leq &\delta \lip{F}+ 6\delta \ninf{F}.
\end{align}
On $E_n$, we also have $|C(t_i) - C(U^\uparrow_{k_i;j})|+  |Z(t_i) - Z(U^\uparrow_{k_i;j})|< \delta$, so it likewise follows that 
\begin{equation}\label{eq:cvgeZ}
\bigg|\E{F\Big(\big(\limC(t_i),\limZ(t_i)\big)_{1\le i \le m}\Big)- F\Big(\big(\limC(U^\uparrow_{k_i;j}),\limZ(U^\uparrow_{k_i;j})\big)_{1\le i \le m}\Big)}\bigg|\leq \delta \lip{F}+ 6\delta \ninf{F}.
\end{equation}
Finally, Lemma~\ref{lem:fdd} implies that 
\begin{equation*}
\qquad\E{F\Big(\big(a_nC_n(U^\uparrow_{k_i;j}),b_nZ_n(U^\uparrow_{k_i;j})\big)_{1\le i \le m}\Big)}\\ \to \E{F\Big(\big(\limC(U^\uparrow_{k_i;j}),\limZ(U^\uparrow_{k_i;j})\big)_{1\le i \le m}\Big)} \qquad
\end{equation*}
as $n\rightarrow \infty$. Together with \eqref{eq:cvgCZ} and \eqref{eq:cvgeZ}, this yields \eqref{eq:toprove} and completes the proof. 
\end{proof}

For use in the next section, we note two further straightforward points related to convergence of processes built from trees. 
For a tree $t$ with $|V(t)|=n\ge 1$, and $q\in S$ a fixed type, define $\Lambda^{(q)}_{t}(i/(2n-2))$ as the number of times the contour process visits for the first time a vertex of type $q$ before time $i$. More formally, for $0\leq i \leq 2n-2$, set 
\[
	\Lambda^{(q)}_{t}\Big(\frac{i}{2n-2}\Big)=\big|\{0\leq j <i\,:\,s(\theta(j))=q\}\big|. 
\]
\nomenclature[Lambda]{$\Lambda_\rt^(q)$}{The process that gives the number of nodes of type $q$ explored before time $s$ in the contour process.}
Then extend the domain of definition of $\Lambda^{(q)}_{t}$ to [0,1] by linear interpolation. The first proposition is a consequence of the fact that if $\rT=(T,d)$ is a random labeled tree with valid law $\nu$, then the law of the underlying plane tree $T$ is symmetric; its proof is omitted. 
\begin{prop}\label{symmetrize_labelproc}
For each $n \ge 1$ let $\rT_n=(T_n,\rD_n)$ be a random labeled multitype plane tree whose law $\nu_n$ is valid, and let $\rT_n^\sym$ have law $\nu_n^\sym$. Fix any type $q \in S$. If there exists a $C([0,1],\R)$-valued random process $\Lambda^{\!(q)}$ such that 
\[
|V(T_n^\sym)|^{-1} \Lambda^{\!(q)}_{T_n^{\sym}} \convdist \Lambda^{\!(q)}\, ,
\]
then also 
\[
|V(T_n)|^{-1} \Lambda^{\!(q)}_{T_n} \convdist \Lambda^{\!(q)}\, . 
\]
\end{prop}
Next, fix a labeled plane tree $\rt=(t,d)$ and list the vertices of $t$ in lexicographic order with respect to their Ulam-Harris labels as $\emptyset=v_0,v_1,\ldots,v_{|t|-1}$, and write $v_{|t|}=v_0$ for convenience. The height process $H_{\rt}$ of $\rt$ is the $C([0,1],\R)$ function defined as follows. 
\nomenclature[ht]{$H_{\rt}$}{The height process of tree $t$.}
For integers $0 \le i \le |t|$ let $H_{\rt}(i/|t|)=\dist_t(\emptyset,v_i)$; then extend to $[0,1]$ by linear interpolation. 
Similarly, define $S_{\rt}:[0,1] \to \R$ by taking $S_{\rt}(i/|t|)=\ell_{\rt}(v_i)$ for $0 \le i \le |t|$ and extending to $[0,1]$ by linear interpolation. 
\begin{prop}\label{prop:height_to_contour}
Let $(\rT_n,n \ge 1)$ be random labelled plane trees. If 
 \[
 (a_nH_{\rT_n},b_nS_{\rT_n}) \convdist (C,Z), 
 \]
for the uniform topology with $a_n \to 0$ and $a_n \cdot |\rT_n| \to \infty$ in probability, then 
 \[
 (a_nC_{\rT_n},b_nZ_{\rT_n}) \convdist (C,Z) 
 \]
 also for the uniform topology. 
\end{prop}
\begin{proof}
We roughly follow the argument from Section~1.6 of \cite{LeGallSurvey}, but must modify it slightly to handle the label process. 
Again fix a labeled plane tree $\rt=(t,d)$ and list the vertices of $t$ in lexicographic order as $v_0,v_1,\ldots,v_{|t|-1}$, writing $v_{|t|}=v_0$ for convenience. For $0 \le i < |t|$ let $j(i)=j_\rt(i) = 2i- \dist_t(\emptyset,v_i)$, 
and set $j(|t|)=2|t|-2$. Also, for $0 \le s \le 1$ let $\varphi_\rt(s) = |t|^{-1}\cdot \sup(i: j_\rt(i) \le  s(2|t|-2))$. 

It is straightforward to verify the following facts.
\begin{itemize}
\item For all $0 \le i < |t|$, $j(i)= \inf(j \ge 0:\theta_\rt(j) = v_i)$.
\item For all $0 \le i < |t|$, for $1 \le k < j(i+1)-j(i)$, the vertex $\theta_\rt(j(i)+k)$ is the parent of $\theta_\rt(j(i)+k-1)$, and if $i < |t|-1$ then $\theta_\rt(j(i+1))$ is the child of $\theta_\rt(j(i+1)-1)$.
In particular, 
\[
\dist_t(\theta_\rt(j(i)),\theta_\rt(j(i)+k))=k\, 
\]
for $1 \le k < j(i+1)-j(i)$. 
\end{itemize}

It follows that if $s \in (0,1)$ and $j(i) \le s(2|t|-2) < j(i+1)$ 
then both $C_\rt(s)$ and $H_\rt(\varphi_\rt(s))$ lie in the interval
\[
[\dist_t(\emptyset,v_{i+1})-1,\dist_t(\emptyset,v_i)+1]\, ,
\]
so
\begin{equation}\label{eq:chbound}
\sup_{0 \le s \le 1} |C_\rt(s)-H_\rt(\varphi_\rt(s))| \le 2+\max_{0 \le i <|t|} \dist_t(v_i,v_{i+1})\, .
\end{equation}
Writing $\alpha_\rt = 2+\max_{0 \le i <|t|} \dist_t(v_i,v_{i+1})$,
it likewise follows that 
\begin{equation}\label{eq:szbound}
\sup_{0 \le s \le 1} |Z_\rt(s)-S_\rt(\varphi(s))| 
\le \max\left(|\ell_\rt(u)-\ell_\rt(v)|: \dist_t(u,v) \le \alpha_\rt\right)\, .
\end{equation}
%
Also, from the definition, we have that
\[
	j_\rt(|t|\varphi_\rt(s))\leq s(2|t|-2)\quad \text{and}\quad j_\rt(|t|\varphi_\rt(s)+1)> s(2|t|-2).
\]  
Thus, for $x > 0$, if 
$\varphi_t(s)-s=x/|t|$ then 
\[
j_{\rt}(|t|s+x) \le s(2|t|-2) 
= 2(|t|s+x)-2(x+1)\, ,
\]
and if $\varphi_t(s)-s=-x/|t|$
then 
\[
j_{\rt}(|t|s-x+1)
> s(2|t|-2) 
= 2(|t|s-x+1)+2(x-2)\, , 
\]
from which it follows that 
\begin{equation}\label{eq:phisbound}
(2|t|-2)\cdot \sup_{0 \le s \le 1} \left| \varphi_\rt(s)-s\right| \le 4+\sup_{0 \le i \le |t|} \left|j(i)-2i\right| = 4+\mathrm{height}(t)\, . 
\end{equation}

We now turn to asymptotics. If $a_nH_{\rT_n} \convdist C$ then the fact that $C$ is continuous implies that
\begin{equation}\label{eq:distTo0}
a_n \cdot \max_{0 \le i < n} \dist_{\rT_n}(v_i,v_{i+1}) \to 0
\end{equation}
in probability, and together with the fact that $a_n|\rT_n|  \convpr \infty$ implies that $\mathrm{height}(\rT_n)/|\rT_n| \to 0$ in probability. 
Using the first of these convergence results in (\ref{eq:chbound}) gives that 
\[
a_n\sup_{0 \le s \le 1} |C_{\rT_n}(s)-H_{\rT_n}(\varphi_{\rT_n}(s))| \to 0
\]
in probability; using the second in (\ref{eq:phisbound}) gives that 
\[
\sup_{0 \le s \le 1} \left| \varphi_{\rT_n}(s)-s\right| \to 0
\]
in probability. Finally, using \eqref{eq:distTo0} in (\ref{eq:szbound}), and exploiting the convergence of $b_nS_{T_n}$ to the continuous process $Z$ as in the proof of Lemma~\ref{lem:labelTight}, it follows that 
\[
b_n\sup_{0 \le s \le 1} |Z_{\rT_n}(s)-S_{\rT_n}(\varphi_{\rT_n}(s))| \to 0
\]
in probability. The last three convergence results together imply that $(a_nH_{\rT_n},b_nS_{\rT_n})$ and $(a_nC_{\rT_n},b_nZ_{\rT_n})$ must have the same limit.
\end{proof}

\section{Convergence of random non-bipartite Boltmann planar maps}\label{sec:planarMaps}
\subsection{Definitions around Theorem~\ref{th:oddAngulations}}
\subsubsection{Brownian tree, Brownian snake, Brownian map}\label{sub:brownian}
In the rest of this section, we denote $\mathbf{e}=(\mathbf{e}(s), 0\leq s\leq 1)$
a standard Brownian excursion. Recall the construction given in Section~\ref{sub:treelike}. The random tree $(\cT_{\mathbf{e}},\Dist_{\mathbf{e}})$ was introduced in~\cite{AldousCRT2} and is called the \emph{Brownian Continuum Random Tree}. 

Next, conditionally given $\mathbf{e}$, let $Z_{\be}=(Z_{\be}(s), 0\leq s \leq 1)$ be a centred Gaussian process such that $Z(0)=0$ and for $0\leq s \leq t \leq 1$, 
\[
	\text{Cov}(Z_{\be}(s),Z_{\be}(t)) = \inf \{\textbf{e}(u), \text{ for } s\leq u\leq t\}.
\]
We may and shall assume $Z_{\be}$ to be a.s. continuous; see \cite[Section~3]{LeGallSurvey} for a more detailed description of the construction of the pair $(\textbf{e},Z_{\be})$, which is called the \emph{Brownian snake}. 
It can be checked that almost surely, for all $x,y \in [0,1]$, if $x\sim_{\mathbf{e}}y$ then $Z_{\be}(x)=Z_{\be}(y)$. Therefore, a.s the pair $(\textbf{e},Z_{\be})$ is a tree-like path. 

To construct the Brownian map, we further need the following. For $x,y\in [0,1]$, let
\begin{equation}\label{eq:defDrond}
D^\circ(x,y)=D^\circ(y,x)=Z_{\be}(x)+Z_{\be}(y)-2\max\big(\min_{q\in [x,y]} Z_{\be}(q),\min_{q\in[y,1]\cup[0,x]}Z_{\be}(q)\big). \end{equation}
Then there is a unique pseudo-distance $D^* \le D^\circ$ on $[0,1]$ such that 
$D^*(x,y)=0$ whenever $x\sim_{\be}y$ which is {\em maximal}, in that if $D' \le D^\circ$ is any other pseudo-distance on $[0,1]$ satisfying this condition, then $D'(x,y) \le D^*(x,y)$ for all $x,y \in [0,1]$. The function $D^*$ is given explicitly as 
\[
D^*(x,y) = \inf\left\{\sum_{i=1}^k D^\circ(x_i,y_i),\right\}
\]
where the infimum is taken over $k \in \N$ and over sequences $(x_1,y_1),\ldots,(x_k,y_k)$ of elements of $[0,1]^2$ with $x_1=x,y_k=y$, and $D^{\circ}(x_{i},y_{i+1})=0$ for $1 \le i < k$; see \cite{MR1835418}, Section 3.1.2, and also \cite{LeGallInventiones,MiermontBrownian}. 

Then, let $M=[0,1]/\{D^*=0\}$ and let $d^*$ be the push-forward of $D^*$ to $M$. Finally, let $\lambda$ be the push-forward of Lebesgue measure on $[0,1]$ to $M$. We refer to the triple $(M,d^*,\lambda)$, or to any other random variable with the same law as $(M,d^*,\lambda)$, as the \emph{Brownian map}.

\subsubsection{Gromov-Hausdorff and Gromov-Hausdorff-Prokhorov distance}
We give in this section the definition of Gromov-Hausdorff and Gromov-Hausdorff-Prokhorov distances and refer the readers to~\cite{GrevenPfaffelhuberWinter} and \cite[Section~6]{Mier09} for details.

Let $X$ and $X'$ be two compact metric spaces. The \emph{Gromov-Hausdorff distance}\footnote{We define here in fact a pseudo-distance and we should consider instead isometric classes of compact metric spaces to be perfectly rigorous.} $\mathrm{d}_{GH}(X,X')$ between $X$ and $X'$ is defined by 
\[
	\mathrm{d}_{GH}(X,X')=\inf_{\phi,\phi'}\delta_H(\phi(X),\phi(X')),
\]
where the infimum is taken over all isometries $\phi:X\rightarrow Z$ and $\phi':X'\rightarrow Z$ into a common metric space $Z$ and where $\delta_H$ denotes the classical Hausdorff distance between closed subsets of $Z$. 

Let $\mathrm{X}=(X,\mu)$ and $\mathrm{X'}=(X',\mu')$ be two compact measured metric spaces (that is $X$ and $X'$ are two compact metric spaces and $\mu$ and $\mu'$ are Borel probability measures on $X$ and $X'$ respectively). The \emph{Gromov-Hausdorff-Prokhorov distance} $\mathrm{d}_{GHP}(\mathrm{X},\mathrm{X'})$ between $\mathrm{X}$ and $\mathrm{X'}$ is defined by 
\[
	\mathrm{d}_{GHP}(\mathrm X, \mathrm{X}')=\inf_{\phi,\phi'}\Big(\delta_H(\phi(X),\phi(X'))\wedge \delta_P(\phi_*\mu,\phi_*\mu')\Big),
\]
where the infimum is taken over all isometries $\phi:X\rightarrow Z$ and $\phi':X'\rightarrow Z$ into a common metric space $Z$ and where $\delta_P$ denotes the Prokhorov distance between two probability measures and $\phi_*\mu$ and $\phi_*\mu'$ denote the push-forwards of $\mu$ and $\mu'$ by $\phi$ and $\phi'$.

\subsection{Planar maps and Boltzmann distribution}\label{sub:maps}
All maps considered in this  section are \emph{rooted}, meaning that an edge is marked and oriented. This edge is called the root edge and its tail is the \emph{root vertex}. In addition to their rooting, maps can also be \emph{pointed} meaning that an additional vertex is distinguished. A map $M$ rooted at an oriented edge $e$ and pointed at a vertex $v^\sbt$ is denoted $(M,e,v^\sbt)$. The set of rooted maps and of rooted and pointed maps are respectively denoted $\mathcal{M}$ and $\mathcal{M}^\bullet$. For $n\geq 1$, we denote by $\mathcal{M}_n$ and $\mathcal{M}^\bullet_n$ the subsets of $\mathcal{M}$ and $\mathcal{M}^\bullet$ consisting of maps with $n$ vertices, respectively. We assume that $\mathcal{M}$ and $\mathcal{M}^\bullet$ both contain the ``vertex map'' $\dagger$, which consists of a single vertex, no edge, and a single face of degree 0.

Following~\cite{MarckertMiermontInvariance} and \cite{MiermontInvariance}, we introduce the \emph{Boltzmann distribution} on $\mathcal{M}$ defined as follows. Let $\mathbf{q}=(q_1,q_2,q_3,\ldots)$ be a sequence of non-negative real numbers.
For $m\in \mathcal{M}$ or $m\in \cM^\bullet$, we define the weight of $m$ by 
\[
W_{\mathbf{q}}(m) = \prod_{f\in f(m)}q_{\deg(f)}, 
\]
where $f(m)$ is the set of faces of $m$; by convention $W_{\mathbf{q}}(\dagger)=1$. 
\medskip

For $n \ge 1$ let $\mathcal{Z}_{\mathbf{q},n}=\sum_{m\in \mathcal{M}_n}W_{\mathbf{q}}(m)$, let $\mathcal{Z}^\bullet_{\mathbf{q},n} = \sum_{m\in \mathcal{M}^\bullet_n} W_{\mathbf{q}}(m) = n \mathcal{Z}_{\mathbf{q},n}$,
and let
\[
\mathcal{Z}^\bullet_{\mathbf{q}}=\sum_{m^\sbt\in \mathcal{M^\bullet}}W_{\mathbf{q}}(m^\sbt) = \sum_{n \ge 1}\sum_{m^\sbt\in \mathcal{M}_n} nW_{\mathbf{q}}(m^\sbt)\, .
\]
From now until the end of the paper, we assume the following holds. 
\begin{assumption}\label{assum:q}
The sequence $\mathbf{q}$ has finite support, and there exists an odd integer $p\geq 3$ such that $q_p>0$. Moreover, $\mathcal{Z}_{\mathbf{q}}^{\bullet}<\infty$. 
\end{assumption}

Under this assumption, we may define probability measures $\mathbb{P}_{\mathbf{q}}$ and $\mathbb{P}^\bullet_{\mathbf{q}}$ on $\cM_n$ and $\cM_n^\bullet$, respectively, by setting, for $m\in \mathcal{M}_n$ and $m^\sbt \in \cM_n^\sbt$, 
\[
	\mathbb{P}_{\mathbf{q},n}(m) = \frac{W_{\mathbf{q}}(m)}{\mathcal{Z}_{\mathbf{q},n}}\quad \text{and}\quad\mathbb{P}^\bullet_{\mathbf{q}}(m^\sbt) = \frac{W_{\mathbf{q}}(m^\sbt)}{\mathcal{Z}^\sbt_{\mathbf{q},n}}.
\]
These definitions only make sense when $\mathcal{Z}_{\mathbf{q},n}\neq 0$, and in what follows, we restrict attention to values of $n$ for which this the case.

Note that, assuming $\mathbf{q}$ has finite support, we may always choose $c > 0$ such that $\cZ(c\mathbf{q})<\infty$, where we write $c\mathbf{q}=(cq_1,cq_2,\ldots)$. 
Moreover, if $\cZ_{\mathbf{q},n}< \infty$ and $\mathcal{Z}_{c\mathbf{q},n}<\infty$, then we have $\mathbb{P}_{\mathbf{q},n}=\mathbb{P}_{c\mathbf{q},n}$ and $\mathbb{P}_{\mathbf{q},n}^\sbt=\mathbb{P}_{c\mathbf{q},n}^\sbt$.
Also, it is proved in the Appendix of \cite{CurienLeGallMiermontCactus} that if $\mathbf{q}$ has finite support and there exists $p\geq 3$ such that $q_p>0$, then there is $c>0$ such that $c\mathbf{q}$ is regular critical. Together these facts imply that, in studying $\mathbb{P}_{\mathbf{q},n}$ and $\mathbb{P}_{\mathbf{q},n}^\sbt$ we may assume $\mathbf{q}$ is itself regular critical; we make this assumption henceforth. 




\medskip
The goal of this section is to prove the following result.
\begin{thm}\label{th:convergence}
Let $(\mathbf{M}_n,n \ge 1)$ be random rooted maps with $\mathbf{M}_n$ having law $\mathbb{P}_{\mathbf{q},n}$. Denote by $\dist_{\bM_n}$ the graph distance on $\bM_n$ and by $\mu_n$ the uniform probability distribution on $V(\bM_n)$. Then there exists a constant $b_\mathbf{q}$ such that, as $n$ goes to infinity, 
\[
	\Big(V(\bM_n),\frac{b_{\mathbf{q}}}{n^{1/4}}d_{\bM_n},\mu_n\Big) \convdist (M,d^*,\lambda),
\]
for the Gromov-Hausdorff-Prokhorov topology, where $(M,d^*,\lambda)$ is the Brownian map.
\end{thm}

\noindent {\bf Remarks.} \\{}
$\star$ If there exists an odd integer $p\geq 3$ such that $q_p > 0 $ and $q_j = 0$ for all $j\neq p$, then clearly $\mathbf{q}$ satisfies Assumption~\ref{assum:q}. In this case $\mathbb{P}_{\mathbf{q},n}$ is the uniform distribution on the set of $p$-angulations with $n$ vertices. Thus, Theorem~\ref{th:oddAngulations} is a direct corollary of Theorem~\ref{th:convergence}. 

\noindent $\star$ The pushforward of $\mathbb{P}^\bullet_{\mathbf{q},n}$ to $\mathcal{M}_n$ obtained by forgetting the marked vertex is $\mathbb{P}_{\mathbf{q},n}$, so we may and will prove Theorem~\ref{th:convergence} for $\bM_n$ distributed according to $\mathbb{P}^\bullet_{\mathbf{q},n}$ rather than $\mathbb{P}_{\mathbf{q},n}$. 
\medskip

To prove this theorem, we rely on the method introduced by Le Gall in \cite[Section 8]{LeGallUniqueness}. This approach exploits distributional symmetries of the Brownian map (which can be deduced from the fact that the ensemble of quadrangulations has the Brownian map as their scaling limit \cite{LeGallUniqueness,MiermontBrownian}) to simplify the task of proving convergence for other ensembles. 

At a high level, to apply the method, three points need to be checked. First, $\bM_n$ must be encoded by a labeled tree such that vertices of the map are in correspondence with a subset of vertices of the tree and such that the labels on the vertices of the tree encode certain metric properties of the map. Second,  the contour and label processes of the labeled trees encoding the sequence of maps should converge (once properly rescaled) to the Brownian snake, (defined in Section~\ref{sub:brownian}). Third,
the vertex with minimum label in the tree must correspond to a vertex in the map whose distribution is asymptotically uniform on $V(\bM_n)$ as $n$ goes to infinity. 

In our setting, the first point is achieved by the Bouttier-Di Francesco-Guitter bijection, which we recall in the next section. The third one, addressed in Section~\ref{sub:invariance}, is a direct consequence of a result by Miermont. Proving that the second point holds (its precise statement will be given in Proposition~\ref{prop:convSnake}) is the main new contribution of this section. 

With these three ingredients, we conclude the proof of Theorem~\ref{th:convergence}~in Section~\ref{sub:convMaps}.

\subsection{The Bouttier-Di Francesco-Guitter bijection}\label{sub:BDG}
In this section, we describe the Bouttier-Di Francesco-Guitter bijection \cite{BDG2004}, roughly following the presentation given in \cite{CurienLeGallMiermontCactus}. 

Let $(M,e,v^\sbt)$ be an element of $\mathcal{M}^\bullet$. Let $e^-$ and $e^+$ be respectively the tail and the head of $e$. Three cases can occur: either $d_M(s,e^-)=d_M(s,e^+)$, $d_M(s,e^-)=d_M(s,e^+)+1$ or $d_M(s,e^-)=d_M(s,e^+)-1$. Depending on which case occurs, we say that $M$ is respectively \emph{null}, \emph{negative} or \emph{positive}. The set of null (resp. positive and negative) pointed and rooted maps is denoted $\mathcal{M}^0$ (resp. $\mathcal{M}^+$ and $\mathcal{M}^-$). By convention, we let $\dagger \in \mathcal{M}^+$.
Reversing the orientation of the root edge gives a bijection between the sets $\mathcal{M}^+\backslash\{\dagger\}$ and $\mathcal{M}^-$. Thus, in the following, we focus only on the sets $\mathcal{M}^+$ and $\mathcal{M}^0$. 
\smallskip

We now introduce the class of decorated trees or \emph{mobiles} which appear in the bijection.
\begin{definition}\label{def:mobiles}
A \emph{mobile} $\rt=(t,d)$ is a 4-type rooted plane labeled tree which satisfies the following constraints. 
\begin{enumerate}[(i)]
	\item Vertices at even generations are of type 1 or 2 and vertices at odd generations are of type 3 or 4.
	\item Each child of a vertex of type 1 is of type 3.
	\item Each non-root vertex of type 2 has exactly one child of type 4 and no other child. If the root vertex is of type 2, it has exactly two children, both of type 4.
\end{enumerate}
The labelling $d$ is an \emph{admissible labeling} of $t$, meaning that the following hold. 
\begin{enumerate}[(1)]
	\item If the root is of type 1, vertices of type 1 and 3 are labeled by integers and vertices of type 2 and 4 by half-integers.
	\item If the root is of type 2, vertices of type 2 and 4 are labeled by integers and vertices of type 1 and 3 by half-integers.
	\item For all vertices $u$ of type 3 or 4, for every $i=0,\ldots,k_t(u)$, 
	\[
	\begin{cases}
	\ell(u(i+1))\geq \ell(ui)-1& \text{ if }u(i+1) \text{ is of type 1},\\
	\ell(u(i+1))\geq \ell(ui)-1/2& \text{ if }u(i+1) \text{ is of type 2}.
	\end{cases}
	\]
	where we use Ulam-Harris encoding together with the convention that $u(k_t(u)+1)$ and $u0$ both denote the parent of $u$.
	\item For all vertices $u$ of type 3 or 4, we have $\ell(u)=\ell(u0)$.
\end{enumerate}
\end{definition}

The set of mobiles such that the root is of type 1 (resp. of type 2) is denoted $\mathbb{T}^+$ (resp. $\mathbb{T}^0$). For $n\geq 1$, the respective subsets of $\mathbb{T}^+$ and $\mathbb{T}^0$ with $n-1$ vertices of type 1 are denoted $\mathbb{T}_n^+$ and $\mathbb{T}^0_n$. We also set $\mathbb{T}^-=\mathbb{T}^+$ and $\mathbb{T}_n^-=\mathbb{T}_n^+$. The notation is justified by Proposition~\ref{prop:BDG}, below. 

\begin{figure}[t!]
\centering
\includegraphics[page=2,width=0.9\textwidth]{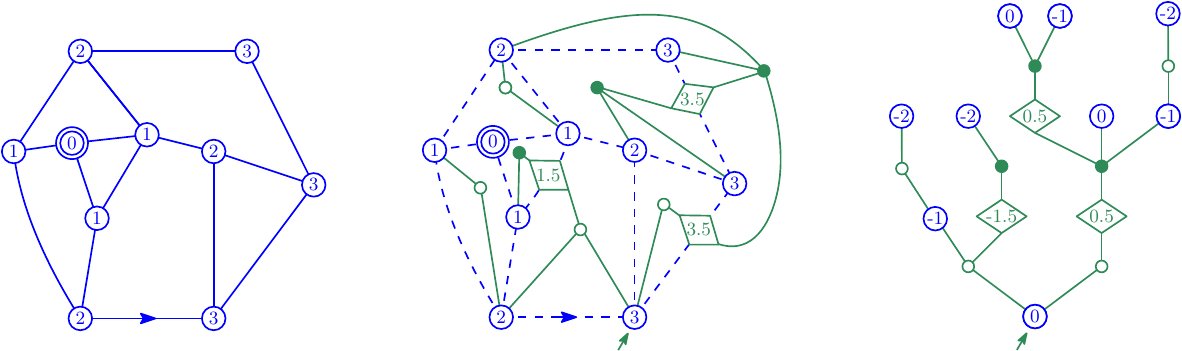}
\caption{\label{fig:BDG} An example of the Bouttier-Di Francesco-Guitter bijection. The original map (left), the construction of $\Phi(m)$ (middle) and the resulting mobile (right). Circle and square labeled vertices correspond respectively to vertices of type 1 and 2. Empty and filled unlabeled vertices correspond respectively to vertices of type 3 and 4.}
\end{figure}
We now give the construction which maps an element of $\mathcal{M}^+$ to an element of $\mathbb{T}^+$ and which is illustrated in Figure~\ref{fig:BDG}.
(the construction for $\mathcal{M}^0$ being very similar, we refer the reader to the original paper~\cite{BDG2004} or to~\cite{CurienLeGallMiermontCactus} for the details).
First, label each vertex of $M$ by $\ell_M(v):=d_M(v,v^\sbt)$. Then, for each edge of the map whose both extremities have the same label, say $\ell$, add a ``flag-vertex'' in the middle of the edge and label it $\ell_M(f):=\ell+1/2$. Call the resulting augmented map $M'$. Next, add a ``face-vertex'' $f$ in each face of the map. Now, for each face of $M'$, considering its vertices in clockwise direction, each time a vertex $v$ is immediately followed by a vertex $w$ with smaller label (the introduction of flag-vertices ensure that any two adjacent vertices have different labels), draw an edge between $v$ and the corresponding face-vertex. Erase all edges of $M'$. 

The result of \cite{BDG2004} ensures that the resulting map, denoted $\Phi(M)$, is in fact a spanning tree of the union of the set of face-vertices, of the set of flag-vertices and of the set of vertices $V(M)\backslash\{v^\sbt\}$. The tree $\Phi(M)$ inherits a planar embedding from $M$. To make it a rooted plane tree, we additionally root it at $e^+$, and choose the first child of $e^+$ to be the face-vertex associated to the face on the left of $(e^-,e^+)$; note that because $M$ is positive, there always exists an edge in $\Phi(M)$ between $e^+$ and this face-vertex.

We assign types to the vertices of $\Phi(M)$ as follows. Vertices of $M$ have type 1, and flag-vertices have type 2. Face-vertices have type 3 if their parent is of type 1 and have type 4 otherwise, This turns $\Phi(M)$ is a mobile, rooted at a vertex of type 1. For $v \in V(M) \setminus \{v^\sbt\}$, by a slight abuse of notation, we also denote $v$ the image of $v$ in $\Phi(M)$.

Label the nodes of $\Phi(M)$ as follows. For $u$ of type 1 or 2, let $\ell(v)=\ell_M(v)-\ell_M(e^+)$, this makes sense since $v$ is a node of $M'$.
Having rooted $\Phi(M)$, we give vertices of type 3 and 4 the same label as their parent. We now use these vertex labels to turn $\Phi(M)$ into a rooted labeled tree by giving each edge $(u,ui)$ of $\Phi(M)$ the label $\ell(ui)-\ell(u)$.
\smallskip

The properties of this construction which are essential to our work appear in the following proposition. (Properties (i) and (ii) are contained in the above description and Property (iii) is Lemma~3.1 of~\cite{LeGallInventiones}).
\begin{prop}[Properties of the Bouttier-Di Francesco-Guitter bijection]\label{prop:BDG}
For each $n\geq 1$ and $\star\in\{-,0,+\}$, $\Phi$ gives a bijection between $\mathcal{M}_n^\star$ and $\mathbb{T}_{n}^\star$. For $(m,e,v^\sbt) \in \mathcal{M}^+$, write $\rt=(t,d)$ for the image of $(m,e,v^\sbt)$ by $\Phi$. Then 
\begin{enumerate}[(i)]
\item Elements of $V(m)\backslash \{v^\sbt\}$ are in bijection with vertices of type 1 in $t$.
\item For all $v\in V(m)\backslash \{v^\sbt\}$, $\dist_m(v,v^\sbt) = \ell_\rt(v)-\min_{x\in V(t)}\ell_\rt(x)+1$.
\item For all $u,v \in V(m)\backslash \{v^\sbt\}$, 
\[
\dist_m(u,v) \leq \ell_\rt(v) + \ell_\rt(u) - 2 \max\Big(
\check \ell_\rt\big(u,v\big),\check \ell_\rt\big(v,u\big)\Big)+2,
\]
where $\check \ell_\rt$ is defined as follows. For $x,y \in \rt$, set $i_x=\inf\{i\,:\,\theta_\rt(i)=x\}$ and $i_y=\inf\{i\,:\,\theta_\rt(i)=y\}$. Then 
\[
	\check\ell_\rt\big(x,y\big)=\begin{cases}
	\inf_{i \in [i_x,i_y]}\ell_t(\theta_\rt(i))&\text{if }i_x\leq i_{y}\\
	\inf_{i \in [i_x,2|\rt|-2]\cup[0,i_y]}\ell_t(\theta_\rt(i))&\text{if }i_y< i_{x}\, .
	\end{cases}
\]
\end{enumerate}
\end{prop}

\subsection{Convergence of labeled trees}\label{sub:invariance}
This section is devoted to the study of random labeled trees obtained by applying the Bouttier-Di Francesco-Guitter bijection $\Phi$ to random maps distributed according to $\mathbb{P}_{\mathbf{q},n}$. 

\begin{thm}\label{thm:convSnake}
There exist $a_{\textbf{q}}>0$ and $b_{\textbf{q}}>0$ such that the following holds. 
For $n \ge 1$ with $Z_{\mathbf{q},n} >0$ let $M_n$ have law $\mathbb{P}_{\mathbf{q},n}$, and let $\rT_n=\Phi(M_n)$ be obtained by applying the  Bouttier-Di Francesco-Guitter bijection to $M_n$. Then as $n \to \infty$ along values with $Z_{\mathbf{q},n} >0$, 
\[
	\left(\frac{a_{\textbf{q}}}{n^{1/2}}C_{\rT_n}(s),\frac{b_{\textbf{q}}}{n^{1/4}}Z_{\rT_n}(s)\right)_{0\leq s \leq 1}  \convdist (\mathbf{e}(s),Z_{\mathbf{e}}(s))_{0\leq s\leq 1},
\]
for the topology of uniform convergence on $C([0,1],\R^2)$.
Moreover, for each $i\in\{1,2,3,4\}$, there exists $\gamma_i\geq 0$ such that 
\[
	\Big(\frac{1}{n}\Lambda^{(i)}_{\rT_n}(t),0\leq s\leq 1\Big)\convdist \big(\gamma_is,0\leq s\leq 1\big)\, ,
\]
for the topology of uniform convergence on $C([0,1],\R)$.
\end{thm}
Theorem~\ref{thm:convSnake} is an immediate consequence of the following proposition. For $n \ge 1$ let $\mathcal{Z}_{\mathbf{q},n}^+ = \sum_{m \in \mathcal{M}_n^+} W_{\mathbf{q}}(m)$ and, if $\mathcal{Z}_{\mathbf{q},n}^+>0$ then define $\mathbb{P}_{\mathbf{q},n}^+$ by 
\[
\mathbb{P}_{\mathbf{q},n}^+(m) = \frac{W_{\mathbf{q}}(m)}{\mathcal{Z}_{\mathbf{q},n}^+}\, .
\]
Likewise define $\mathcal{Z}_{\mathbf{q},n}^0,\mathcal{Z}_{\mathbf{q},n}^-$ and $\mathbb{P}_{\mathbf{q},n}^0$,$\mathbb{P}_{\mathbf{q},n}^-$. 
\begin{prop}\label{prop:convSnake}
There exist $a_{\textbf{q}}>0$ and $b_{\textbf{q}}>0$ such that for any symbol $\star \in \{-,0,+\}$ the following holds. 
For $n \ge 1$ with $Z_{\mathbf{q},n}^\star >0$ let $M_n^\star$ have law $\mathbb{P}_{\mathbf{q},n}^\star$, and let $\rT_n^\star=\Phi(M_n^\star)$ be obtained by applying the  Bouttier-Di Francesco-Guitter bijection to $M_n^\star$. Then as $n \to \infty$ along values with $Z_{\mathbf{q},n}^\star >0$, 
\[
	\left(\frac{a_{\textbf{q}}}{n^{1/2}}C_{\rT_n^\star}(s),\frac{b_{\textbf{q}}}{n^{1/4}}Z_{\rT_n^\star}(s)\right)_{0\leq s \leq 1}  \convdist (\mathbf{e}(t),Z_{\mathbf{e}}(s))_{0\leq s\leq 1},
\]
for the topology of uniform convergence on $C([0,1],\R^2)$.
Moreover, for each $i\in\{1,2,3,4\}$, there exists $\gamma_i\geq 0$ such that 
\[
	\Big(\frac{1}{|\rT^\star_n|}\Lambda^{(i)}_{\rT^\star_n}(s),0\leq s\leq 1\Big)\convdist\big(\gamma_is,,0\leq s\leq 1\big)\, ,
\]
for the topology of uniform convergence on $C([0,1],\R)$.
\end{prop}
\begin{proof}
We provide details only for the case that $\star=+$, and briefly discuss the other cases at the end of the proof. 
Using the notation of Section~\ref{sub:validAndGW}, Proposition~4.6 of~\cite{CurienLeGallMiermontCactus} 
gives the following  description of the distribution of $\Phi(M_n^+)=\rT^+_n=(T^+_n,D^+_n)$. 
\begin{enumerate}
	\item The tree $T^+_n$ is a 4-type Galton-Watson tree whose root has type 1 and which is conditioned to have $n-1$ vertices of type 1. Its offspring distribution $\bm{\zeta}^+_{\mathbf{q}}$ is as follows. 
	\begin{itemize}
	\item The support of $\zeta^{(1)}_{\mathbf{q}}$ is $\{0\}\times\{0\}\times\mathbb{Z}_+\times\{0\}$, and for $k\geq 0$, 
	\[
\zeta^{(1)}_{\mathbf{q}}(0,0,k,0)=\frac{1}{\mathcal{Z}_{\mathbf{q}}^+}\Big(1-\frac{1}{\mathcal{Z}_{\mathbf{q}}^+}\Big)^k.\]
	\item $\zeta^{(2)}_{\mathbf{q}}(0,0,0,1)=1$.
	\item $\zeta^{(3)}_{\mathbf{q}}$ and $\zeta^{(4)}_{\mathbf{q}}$ are supported on $\mathbb{Z}_+\times\mathbb{Z}_+\times\{0\}\times\{0\}$, and for any $k,k'\geq 0$,
	\begin{align*}
	\zeta^{(3)}_{\mathbf{q}}(k,k',0,0)&=\alpha_{\mathbf{q}}({\mathcal{Z}_{\mathbf{q}}^+})^k({\mathcal{Z}_{\mathbf{q}}^0})^{k'/2}\binom{2k+k'+1}{k+1,k,k'}q_{2k+k'+2}\\
	\zeta^{(4)}_{\mathbf{q}}(k,k',0,0)&=\beta_{\mathbf{q}}({\mathcal{Z}_{\mathbf{q}}^+})^k({\mathcal{Z}_{\mathbf{q}}^0})^{k'/2}\binom{2k+k'}{k,k,k'}q_{2k+k'+1},
	\end{align*}
	where $\alpha_{\mathbf{q}}$ and $\beta_{\mathbf{q}}$ are the appropriate normalizing constants.
	\end{itemize}
	\item Conditionally given $T^+_n$, the labeling $D^+_n$ is uniformly distributed over all admissible labelings (see Definition~\ref{def:mobiles}). 
\end{enumerate}
The law $\nu^+_{\mathbf{q},n}$ of $\rT_n^+$ is valid (see Section~\ref{sub:validAndGW}) but its displacements are not locally centered. However, by Lemma~2 of \cite{MiermontInvariance}, we know that its displacements are centered. Moreover, it follows directly from Proposition~3 of \cite{MiermontInvariance} that the symmetrization $\rT_n^{+,\sym}$ satisfies all the assumptions of Theorems~2 and~4 of~\cite{MiermontInvarianceTrees}\footnote{N.B.\ The term ``centered'' as used in \cite[Theorem~4]{MiermontInvarianceTrees} in fact corresponds to the notion of ``locally centered'' used in the current work.}.
The conclusion of those theorems is that there exist $a_{\textbf{q}}>0$ and $b_{\textbf{q}}>0$ such that, as $n \to \infty$ along values with $\mathcal{Z}_{\mathbf{q},n}^+ > 0$, with $\rT_{n}^{+,\sym}$ distributed as $\nu^{+,\sym}_{\mathbf{q},n}$, then 
\[
\left(\frac{a_{\textbf{q}}}{n^{1/2}}H_{T^{+,\sym}_n}(s),\frac{b_{\textbf{q}}} {n^{1/4}}S_{\rT^{+,\sym}_n}(s)\right)_{0\leq s \leq 1}  \convdist (\mathbf{e} (s),Z(s))_{0\leq s\leq 1}, 
\] 
for the topology of uniform convergence on $\mathcal{C}([0,1]^2)$, and for all $i \in \{1,2,3,4\}$, 
\[
	\Big(\frac{1}{n}\Lambda^{(i)}_{\rT^{+,\sym}_n}(s),0\leq s\leq 1\Big)\convpr \big(\gamma_is,,0\leq s\leq 1\big)\, 
\]
for the topology of uniform convergence on $\mathcal{C}([0,1])$, for suitable constants $(\gamma_i,1 \le i \le 4)$. 
By Proposition~\ref{prop:height_to_contour} and  Theorem~\ref{thm:main}, the first convergence also holds with $H_{T^{+,\sym}_n}$ and $S_{T^{+,\sym}_n}$ replaced by $C_{T^{+}_n}$ and $Z_{T^{+}_n}$, respectively; by Proposition~\ref{symmetrize_labelproc}, the second convergence also holds  with $T^{+,\sym}_n$ replaced by $T^{+}_n$. This completes the proof in the case that $\star=+$. 

Since it suffices to flip the orientation of the root edge of a map sampled from $\mathbb{P}^+_{\mathbf{q},n}$ to simulate $\mathbb{P}^-_{\mathbf{q},n}$, the case that $\star=-$ reduces to the case that $\star=+$. 
In the case $\star=0$, the pushforward $\nu^0_{\mathbf{q},n}$ of $\mathbb{P}^0_{\mathbf{q},n}$ by $\Phi$ is very similar to $\nu^+_{\mathbf{n,q}}$, the only exception being that the root has type 2 and has two children of type 4. 
We are hence left with an ordered pair of Galton-Watson trees, both with root vertex of type 4, and conditioned to together contain a total of $n-1$ vertices of type 1. A reprise of the above arguments (again using the results of~\cite{MiermontInvarianceTrees} together with Theorem~\ref{thm:main}) again yields the result in this case. (The fact that there are two trees rather than one may seem like an issue; but it is known that in this case one of the two trees will have $\mathcal{O}(1)$ size and will disappear in the limit, while the other one will exhibit the desired invariance principle. The extension of the invariance principle from trees to forests with a bounded number of trees is explained in the remark on page 1149 of \cite{MiermontInvarianceTrees}.)

\end{proof}

\subsection{Convergence of Boltzmann maps}\label{sub:convMaps}
We conclude the proof of Theorem~\ref{th:convergence} in this section. Our argument closely mimics that given in~\cite[Section~8]{LeGallUniqueness} for the convergence of triangulations, so we only give the main steps of the proof. 
Let $\mathbf{M}_n$ have law $\mathbb{P}^+_{\mathbf{q},n}$. We shall prove that
\[
	\Big(V(\mathbf{M}_n),\frac{b_{\mathbf{q}}}{n^{1/4}}\dist_{\mathbf{M}_n},\mu_n\Big) \convdist (M,d^*,\lambda) 
\]
in distribution for the Gromov-Hausdorff-Prokhorov topology, as $n \to \infty$. 
The same holds with $\mathbf{M}_n$ having law $\mathbb{P}^0_{\mathbf{q},n}$ or $\mathbb{P}^-_{\mathbf{q},n}$, with essentially the same proof (with some minor modifications, as at the end of Section~\ref{sub:invariance}; we omit the details for these cases).

Fix $(m,e,v)$ in $\cM^+$ and write $\Phi(m,e,v)=\rt=(t,d)\in \cT^+$. 
For $i\leq j \in\{0,\ldots,2|t|-2\}$ such that $\theta(i)$ and $\theta(j)$ are both of type 1, we define
\[\dfun_\rt(i,j)=\dist_{m}\big(\theta(i),\theta(j)\big)
\]
and, writing $\ell=\ell_\rt$ for the vertex labeling of $\rt$, 
\[
\dfun_\rt^{\circ}(i,j)=\dfun_\rt^{\circ}(j,i)=\ell(\theta(i))+\ell(\theta(j))-2\max\big(\min_{k\in\{i,\ldots,j\}}\ell(\theta(k)),\min_{k\in\{j,\ldots,2|t|-2\}\cup\{0,\ldots i\}}\ell(\theta(k))\big)+2.
\]
It follows directly from $(iii)$ of Proposition~\ref{prop:BDG} that 
\begin{equation}\label{eq:ddrond}
	\dfun_\rt(i,j)\leq \dfun_\rt^{\circ}(i,j). 
\end{equation} 
Finally, we extend both $\dfun_t$ and $\dfun^\circ_t$ to $[0,2|t|-2]\times [0,2|t|-2]$ by linear interpolation. The inequality in~\eqref{eq:ddrond} readily extends to this whole set.
\smallskip

We write $\bT_n=\Phi(\bM_n)$ and $m_n=2|\bT_n|-2$, and for $0\leq r,s \leq 1$ we define 
\[
	D_n(r,s):=\frac{b_{\mathbf{q}}}{n^{1/4}}\dfun_{\bT_n}(m_nr,m_ns)\quad \text{and}\quad D^\circ_n(r,s):=\frac{b_{\mathbf{q}}}{n^{1/4}}\dfun^\circ_{\bT_n}(m_nr,m_ns).
\]

Recall the definition of $D^\circ$ given in~\eqref{eq:defDrond}. Since $D^\circ$ depends continuously on $Z$, Proposition~\ref{prop:convSnake} implies that $D_n^\circ \convdist D^\circ$; moreover this convergence holds jointly with that stated in Proposition~\ref{prop:convSnake}. Together with the bound \eqref{eq:ddrond} this implies (see~\cite[Proposition~3.2]{LeGallInventiones}) that the family of laws of $(D_n,n \ge 1)$ is tight in the space of probability measures on $C([0,1],\R^2)$. Hence, from any increasing sequence of positive integers, we can extract an increasing subsequence $(n_j)_{j\geq 1}$, such that, jointly with the convergence in Proposition~\ref{prop:convSnake}, we have
\begin{equation}\label{eq:subseq}
	\Big(D^\circ_{n_j}(s,t),D_{n_j}(s,t)\Big)_{0\leq s\leq t \leq 1}\convdist (D^\circ(s,t),\tilde D(s,t))_{0\leq s\leq t\leq 1},
\end{equation}
for some random process $\tilde{D}$ taking values in $C([0,1],\R)$. We will show that this convergence holds without extracting a subsequence with $\tilde D=D^*$; from this the theorem follows just as in~\cite{LeGallUniqueness}, since $d^*$ is the push-forward of $D^*$ to $[0,1]\backslash \{D^*=0\}$
By the Skorohod representation theorem, we may and will assume that the convergence along $(n_j)_{j\geq 1}$ in~\eqref{eq:subseq} and in Proposition~\ref{prop:convSnake} jointly hold almost surely.
To prove that the distributional convergence holds without extracting a subsequence, it then suffices to prove that $\tilde D\aseq D^*$, where $D^*$ is defined in Section~\ref{sub:brownian}; recall that $D^*$ is a measurable function of $({\bf e},D^\circ)$.

By \eqref{eq:ddrond}, necessarily $\tilde D\leq D^\circ$ almost surely. Since $\tilde D$ satisfies the triangle inequality, it then follows from the definition of $D^*$ that $\tilde D\leq D^*$ almost surely. By continuity, to show that $\tilde D\aseq D^*$ it is then enough to prove that for any $U,V$ two independent uniform random variables in $[0,1]$, independent of all the other random objects, we have $\tilde D(U,V)\eqdist D^*(U,V)$.

So let $U,V$ be two such uniform random variables. List the vertices of type 1 in $\rT_n$ in lexicographic order as $v_1,\ldots, v_{n-1}$. For $k\in\{1,\ldots,n-1\}$, let $i_{\rT_n}(k)=\inf\{i\geq 0,\, \theta_{\rT_n}(i)=v_k\}$. Next, define $U_n=\lceil U\cdot(n-1)\rceil$ and $V_n=\lceil V\cdot(n-1)\rceil$ which are both uniformly distributed on $\{1,\ldots,n-1\}$. It follows from the last statement of Proposition~\ref{prop:convSnake} that $(n-1)/|\bT_n| \convpr \gamma_1$ and that
\[
	\frac{i_{\rT_n}(U_n)}{m_n}\convpr U \qquad\text{and} \qquad \frac{i_{\rT_n}(V_n)}{m_n}\convpr V; 
\]
the last statement of Proposition~\ref{prop:convSnake} is written as convergence in distribution, but the limit is non-random and so we indeed obtain convergence in probability. 
Hence, 
\begin{equation}\label{eq:convpr}
	\big|D_n(U,V)-\frac{b_{\mathbf{q}}}{n^{1/4}}\dist_{M_n}(v_{U_n},v_{V_n})\big|\convpr 0.
\end{equation}
	Let now $X_n,Y_n$ be two independent uniform vertices of $\mathbf{M}_n$. Then, the convergence results stated in~\eqref{eq:subseq} and in~\eqref{eq:convpr} together imply that$\dfrac{b_{\mathbf{q}}}{n_j^{1/4}}\dist_{\mathbf{M}_{n_j}}(X_{n_j},Y_{n_j})$ also converges in distribution\footnote{Observe that $v_{U_n}$ and $v_{V_n}$ are two uniform vertices of $\mathbf{M}_n\backslash\{v^\sbt\}$ rather than two uniform vertices of $\mathbf{M}_n$. But, since $|V(\mathbf{M}_n)|=n$, the difference is asymptotically negligible.}.  
	Since, $\dist_{\mathbf{M}_n}(X_n,Y_n)\eqdist\dist_{\mathbf{M}_n}\big(X_n,v^\sbt(\mathbf{M}_n)\big)$, for all $n$, it follows that the limiting distribution of $\dfrac{b_{\mathbf{q}}}{n_j^{1/4}}\dist_{\mathbf{M}_{n_j}}(X_{n_j},Y_{n_j})$ is the same as the limiting distribution of 
	\[
	\frac{b_{\mathbf{q}}}{n^{1/4}}\dist_{\mathbf{M}_n}\big(X_n,v^\sbt(\mathbf{M}_n)\big)=\frac{b_{\mathbf{q}}}{n^{1/4}}\big(Z_{\rT_n}(U_n)-\min Z_{\rT_n}+1\big),
\]
where the equality is a consequence of $(ii)$ in Proposition~\ref{prop:BDG} and of the definition of $Z_{\rT_n}$ (see Section~\ref{sub:def}). 
 This last quantity converges to $Z_{\mathbf{e}}(U)-\inf(Z_{\mathbf{e}}(s),0\leq s \leq 1)$. By~\cite[Corollary~7.3]{LeGallUniqueness}, $Z_{\mathbf{e}}
 (U)-\inf(Z_{\mathbf{e}}(s),0\leq s \leq 1)\eqdist D^*(U,V)$, and therefore $\tilde D\aseq D^*$ as required. 
 \hfill \qed


\addtocontents{toc}{\SkipTocEntry} 
\section{Acknowledgements}
Both authors thank two anonymous referees, whose comments significantly improved the presentation of this work.
LAB was supported throughout this research by an NSERC Discovery Grant, and for part of this research by an FRQNT team grant. MA was supported by an ANR grant under the agreement ANR-16-CE40-0009-01 (ANR GATO).
The final stages of this work were conducted at the McGill Bellairs Research Institute; both LAB and MA thank Bellairs for providing a productive and flexible work environment. 

\small
\printnomenclature[3.1cm]
\normalsize


\end{document}